\documentclass{amsart}

\usepackage{mathtools,amssymb}
\usepackage[pdftex]{hyperref}
\usepackage{xcolor}
\usepackage{graphicx}
\usepackage{ytableau,verbatim,shuffle}
\usepackage{mathrsfs}
\usepackage{enumitem}
\usepackage{tikz}
\usetikzlibrary{calc, shapes, backgrounds,arrows,positioning,plotmarks}

\theoremstyle{plain}

\newtheorem{theorem}[equation]{Theorem}
\newtheorem{lemma}[equation]{Lemma}
\newtheorem{conjecture}[equation]{Conjecture}

\newtheorem{proposition}[equation]{Proposition}

\newtheorem{corollary}[equation]{Corollary}

\theoremstyle{definition}
\newtheorem{definition}[equation]{Definition}
\newtheorem{example}[equation]{Example}
\newtheorem{remark}[equation]{Remark}
\numberwithin{equation}{section}

\newcommand{\Z}{\mathbb Z}

\newcommand{\C}{\mathbb C}

\newcommand{\Gr}{\mathrm{Gr}}

\newcommand{\jp}{\mathrm{jp}}
\newcommand{\word}{\mathrm{word}}

\newcommand{\nc}{\mathcal{NC}}
\newcommand{\ncop}{\mathcal{NCOP}}
\newcommand{\rstab}{\mathcal{J}}
\newcommand{\rstabgen}{\overline{\mathcal{J}}}
\newcommand{\rspoly}{\mathrm{J}}
\newcommand{\inv}{\mathrm{inv}}
\newcommand{\sgn}{\mathrm{sgn}}
\newcommand{\setpart}{\Pi}
\newcommand{\ordsetpart}{\mathcal{OP}}

\DeclareMathOperator{\codim}{\mathrm{codim}}
\DeclareMathOperator{\Proj}{\mathrm{Proj}}

\newcommand{\SL}{\mathrm{SL}}

\definecolor{darkblue}{rgb}{0.0,0,0.7}

\newcommand{\newword}[1]{\textcolor{darkblue}{\textbf{\emph{#1}}}}

\title{Web invariants for flamingo Specht modules}
\author{Chris Fraser}
\address[CF]{\parbox{\linewidth}{Dept.\ of Mathematics, Michigan State University, East Lansing, MI, 48824, USA}}
\email{\parbox[t]{\linewidth}{cmfra@umich.edu}}
\author{Rebecca Patrias}
\address[RP]{\parbox{\linewidth}{Dept.\ of Mathematics, University of St.\ Thomas, St.\ Paul, MN, 55105, USA}}
\email{\parbox[t]{\linewidth}{rebecca.patrias@stthomas.edu}}
\author{Oliver Pechenik}
\address[OP]{\parbox{\linewidth}{Dept.\ of Combinatorics \& Optimization, University of Waterloo, Waterloo, ON, N2L 3G1, Canada}}
\email{\parbox[t]{\linewidth}{oliver.pechenik@uwaterloo.ca}}
\author{Jessica Striker}
\address[JS]{\parbox{\linewidth}{Dept.\ of Mathematics, North Dakota State University, Fargo, ND, 58108, USA}}
\email{\parbox[t]{\linewidth}{jessica.striker@ndsu.edu}}

\begin{document}

\begin{abstract}
Webs yield an especially important realization of certain Specht modules, irreducible representations of symmetric groups, as they provide a pictorial basis with a convenient diagrammatic calculus. In recent work, the last three authors associated polynomials to noncrossing partitions without singleton blocks, so that the corresponding polynomials form a web basis of the pennant Specht module $S^{(d,d,1^{n-2d})}$. These polynomials were interpreted as global sections of a line bundle on a $2$-step partial flag variety. 

Here, we both simplify and extend this construction. On the one hand, we show that these polynomials can alternatively be situated in the homogeneous coordinate ring of a Grassmannian, instead of a $2$-step partial flag variety, and can be realized as tensor invariants of classical (but highly nonplanar) tensor diagrams. On the other hand, we extend these ideas from the pennant Specht module $S^{(d,d,1^{n-2d})}$ to more general flamingo Specht modules $S^{(d^r,1^{n-rd})}$. In the hook case $r=1$, we obtain a spanning set that can be restricted to a basis in various ways. In the case $r>2$, we obtain a basis of a well-behaved subspace of $S^{(d^r,1^{n-rd})}$, but not of the entire module.
\end{abstract}

\maketitle

\section{Introduction}\label{sec:intro}

The irreducible representations of the symmetric group $\mathfrak{S}_n$ are the \emph{Specht modules} $S^\lambda$, indexed by integer partitions $\lambda = (\lambda_1 \geq \lambda_2 \geq \dots \geq \lambda_k > 0)$. There are a variety of ways to construct these modules concretely, each with various pros and cons. One potential virtue of a construction of $S^\lambda$ is that it yields a natural choice of basis with useful properties. Our interest is in obtaining \emph{web bases}, extending seminal work of G.~Kuperberg \cite{Kuperberg}. Various authors differ in precisely what properties they expect a web basis to satisfy. For us, a web basis is one such that 
\begin{enumerate}[label=\textrm{(W.\arabic*)}]
    \item each basis element is indexed by a planar diagram with $n$ boundary vertices, embedded in a disk;\label{W_diagram}
    \item there is a topological criterion allowing identification of basis diagrams; \label{W_alg}
    \item the long cycle $c = (12\ldots n)$ acts on the basis by rotation of diagrams (up to signs); \label{W_c}
    \item \label{W_w0} the long element $w_0 = n(n-1) \ldots 1$ acts on the basis by reflection of diagrams (up to signs); and 
    \item there is a finite list of `skein relations' describing the action of a simple transposition $s_i$ on a basis diagram. \label{W_skein}
\end{enumerate}
Web bases satisfying all these properties are not known for general partitions $\lambda$. (Bases indexed by planar diagrams appear in \cite{Fontaine,Westbury,Elias}; however, these bases do not satisfy properties \ref{W_alg}, \ref{W_c}, or \ref{W_w0}, and indeed in some cases are not even well-defined, depending on arbitrary choices.) For further results on the combinatorics of web bases for particular Specht modules, see, e.g., \cite{Russell.Tymoczko:sl2, Rhoades:polytabloid,Russell.Tymoczko:sl3,Im.Zhu,Hwang.Jang.Oh,Purbhoo.Wu}. For connections to cluster algebras, see, e.g., \cite{Fomin.Pylyavskyy,Fraser.Pylyavskyy}. We briefly review the supply of known web bases, in the sense of properties \ref{W_diagram}--\ref{W_skein}.

\begin{itemize}
    \item 
For $\lambda = (d,d)$, a web basis is given by the set of noncrossing perfect matchings of $n$ vertices. That is, if we place vertices labeled $1, \dots, n$ around the boundary of a disk, web basis elements correspond to ways to pair the vertices by pairwise nonintersecting arcs embedded in the disk. Various aspects of this web basis were known to various authors at various times; see, in particular, \cite{Kung.Rota, Kuperberg, Patrias.Pechenik:evacuation, Petersen.Pylyavskyy.Rhoades, Rhoades:thesis, Temperley.Lieb} for discussion.

\item For $\lambda =(d,d,d)$, a web basis was first constructed by Kuperberg \cite{Kuperberg}, who also established properties \ref{W_diagram}, \ref{W_alg}, and \ref{W_skein} for it. Property~\ref{W_c} was later proven in \cite{Petersen.Pylyavskyy.Rhoades}; property~\ref{W_w0} was first explicitly proven in \cite{Patrias.Pechenik:evacuation}, although known to experts earlier.

\item For $\lambda =(d,d,d,d)$, a web basis satisfying all five properties was recently constructed by the last two authors in joint work with C.~Gaetz, S.~Pfannerer, and J.~Swanson~\cite{Gaetz.Pechenik.Pfannerer.Striker.Swanson:SL4}.

\item For $\lambda = (2^k)$, the first author \cite{Fraser:2column} gave a web basis by showing that the \emph{dual canonical basis} of \cite{Lusztig} can be rendered diagrammatic in this case. Further discussion of this `$2$-column' web basis appears in \cite{Gaetz.Pechenik.Pfannerer.Striker.Swanson:2column}.

\item In \cite{Rhoades:skein} (building on combinatorics from \cite{Pechenik:CSP}), B.~Rhoades gave a web basis for the two-parameter family $\lambda = (d,d,1^{n-2d})$ of \emph{pennant} shapes, extending the noncrossing matching basis from the case $\lambda = (d,d)$. In this case, the planar diagrams are noncrossing set partitions without singleton blocks. This pennant web basis was realized in \cite{Kim.Rhoades} as a space of \emph{fermionic diagonal coinvariants}, and was realized by the last three authors \cite{Patrias.Pechenik.Striker} as a basis of {\sl jellyfish invariants} in the homogeneous coordinate ring of a $2$-step \emph{partial flag variety}. Yet another realization of the pennant web basis was given by \cite{Kim:embedding} as a submodule of an induction product built from the web basis of the case $\lambda = (d,d)$.
\end{itemize}

The goals of this paper are twofold. As our first main goal, we reinterpret the bases from the previous paragraph in terms of the traditional web diagram formalism which underlies the  rectangular cases. In \cite{Patrias.Pechenik.Striker}, the passage from a set partition (a combinatorial object) to a {\sl jellyfish invariant} (an algebraic object) was a {\sl definition}, ``pulled from thin air'' (see e.g. Definition~\ref{def:polynomial}). In the present paper, we work with a copy of the pennant Specht module in the homogeneous coordinate ring of a Grassmannian, rather than a 2-step partial flag variety. This transferal process from a partial flag variety to a Grassmannian is likely known to experts; however, it is difficult to find an explicit description in the literature, so we explain the details of this construction. We then associate to any set partition a certain tensor diagram whose corresponding tensor diagram invariant {\sl is} the jellyfish invariant from \cite{Patrias.Pechenik.Striker}. That is, the passage from jellyfish web to jellyfish invariant is the ``classical'' passage from tensor diagram to tensor diagram invariant. See Figures~\ref{fig:schematic} and \ref{fig:running_tensor} for an illustration of this construction. 

We see several advantages of this change in perspective. First, it becomes easy to verify that jellyfish invariants lie in the pennant Specht module (cf.~Theorem~\ref{thm:inspecht}), which was demonstrated by lengthy calculation in {\sl loc. cit.}. Second, it explains the ``origin'' of the signs in the definition of jellyfish tableaux, as these signs are ``forced'' by the tensor diagram formalism. Third, it allows for the use of classical ${\rm SL}_n$ skein relations and the classical Pl\"ucker relations as tools for proving algebraic identities involving jellyfish invariants. Let us point out, however, that the tensor diagrams arising from the Grassmmanian perspective are quite large and highly nonplanar.

As our second main goal, we extend these new constructions and those of \cite{Patrias.Pechenik.Striker} to the $3$-parameter family of Specht modules $S^{(d^r,1^{n-r\cdot d})}$. We call these \emph{flamingo} Specht modules, since their defining feature is the first long column (cf.~Figure~\ref{fig:flamingo}); pennant Specht modules are the case $r=2$. Outside of the pennant case, our invariants do not directly yield bases of the flamingo Specht module. Rather in the `hook case' $r=1$, we obtain a spanning set of diagrams and invariants. In Section~\ref{sec:hook}, we explain a recipe for choosing linearly independent subsets of this hook basis; however, it is not possible to do so while maintaining property \ref{W_c}. In the remaining case $r>2$, we obtain a linearly independent set of diagrams and invariants, but we do not know a systematic way to extend this set to a basis. Nonetheless, the subspace spanned by these diagrams exhibits various nice properties. For example, it is invariant under the actions of $c$ and $w_0$, with those elements acting by rotation and reflection as in \ref{W_c} and \ref{W_w0}.

\begin{figure}[ht]
    \centering
    \raisebox{1.7in}{\ydiagram{6,6,6,1,1,1,1,1}}
    \includegraphics[scale=0.05,trim={0 20cm 10cm 15cm},clip]{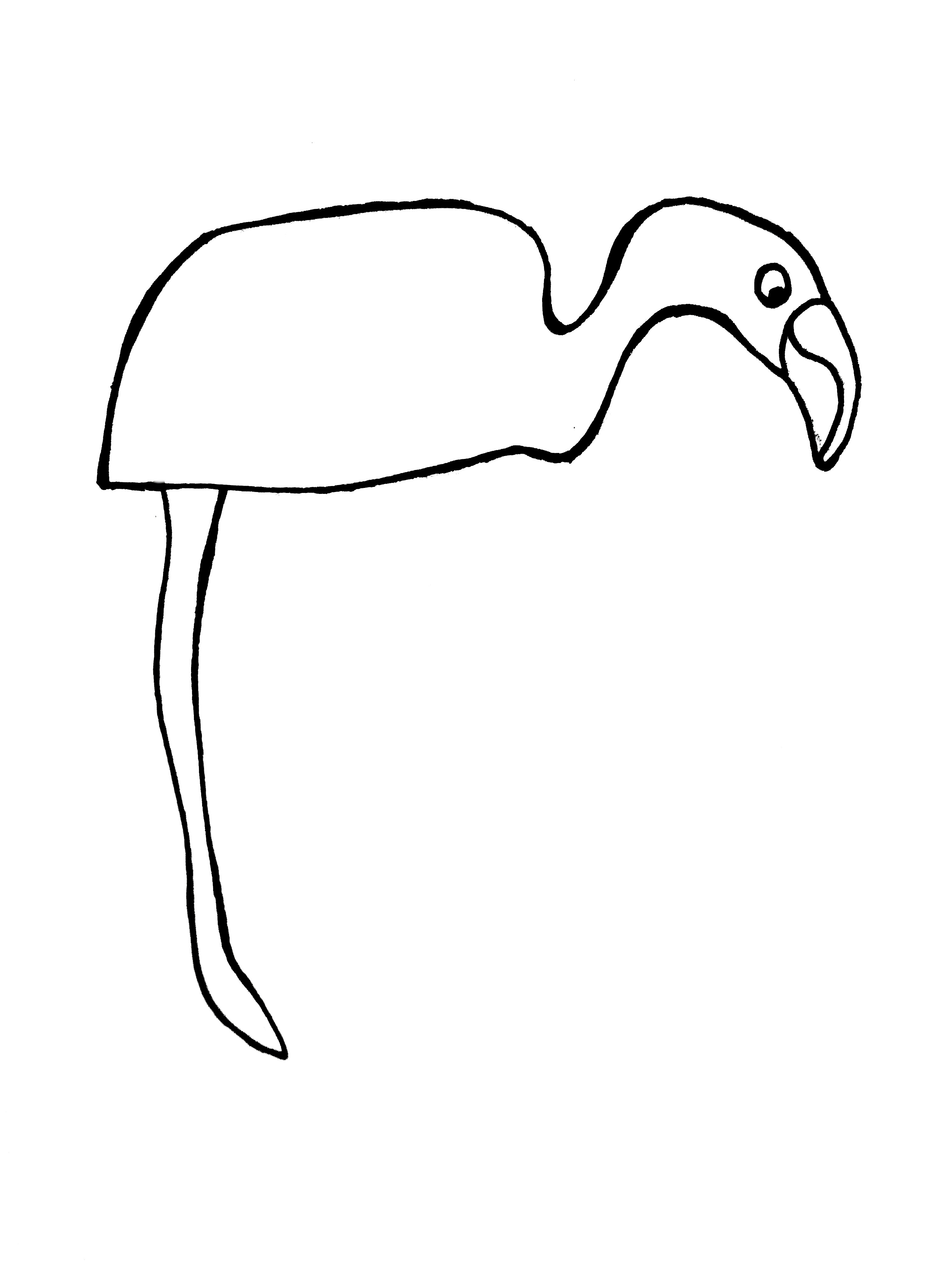}
    \caption{The Young diagram (left) of the flamingo partition $(6^3,1^5)$ and a flamingo (right) blending in.}
    \label{fig:flamingo}
\end{figure}

{\bf This paper is organized as follows.} In Section~\ref{sec:background}, we recall necessary background material. In Section~\ref{sec:invariants}, we extend the definition of \emph{jellyfish tableaux} and \emph{jellyfish invariants}, introduced in \cite{Patrias.Pechenik.Striker}, from $r=2$ to general $r$. Jellyfish invariants are polynomials that we associate to ordered set partitions. In Section~\ref{sec:GrassmannCayley}, we use ideas about Grassmann--Cayley algebras and tensor diagrams to realize jellyfish invariants as elements of the homogeneous coordinate ring of a Grassmannian. We then observe that jellyfish invariants are elements of flamingo Specht modules $S^{(d^r,1^{n-r\cdot d})}$. In Section~\ref{sec:diagrammatics}, we develop a diagrammatic calculus for ordered set partitions and the corresponding jellyfish invariants. Our key technical tool is a $(2^r+1)$-term recurrence relation established in Section~\ref{sec:5term}. Using this recurrence, we establish skein relations for jellyfish invariants, as well as diagrammatic characterizations of the action of $c$ and $w_0$. In Theorem~\ref{thm:linind}, we show that jellyfish invariants are linearly independent in the cases $r>1$.  Section~\ref{sec:hook} is devoted to the distinctive case $r=1$. Finally, Section~\ref{sec:final} contains some remarks on further relations among jellyfish invariants and a conjectural extension of Theorem~\ref{thm:linind}.

\section{Background}\label{sec:background}
\subsection{Ordered set partitions}\label{sec:set_partitions}
We write $[n] = \{1, 2, \dots, n\}$. An \newword{ordered set partition} of $n$ is a sequence $\pi = (\pi_1, \pi_2, \dots, \pi_d)$ of sets where 
\begin{itemize}
    \item each $\pi_i \neq \emptyset$,
    \item $\bigcup_{i} \pi_i = [n]$, and
    \item $\pi_i \cap \pi_j = \emptyset$ if $i \neq j$.
\end{itemize}
We write such an ordered set partition as 
$(\pi_1 \mid \pi_2 \mid \dots \mid \pi_d)$ for visual distinctiveness. The sets $\pi_i$ are called the \newword{blocks} of the ordered set partition. If we forget the ordering of the blocks in an ordered set partition, the result is an \newword{(unordered) set partition} 
$\{\pi_1, \dots, \pi_d \}$.

We draw a set partition $\pi = \{\pi_1, \dots, \pi_d \}$ of $n$ by placing dots labeled $1, 2, \dots n$ clockwise around the boundary of a disk and then, for each $\pi_i$, drawing the convex hull of the boundary dots whose labels are in $\pi_i$. If these convex hulls do no intersect, we call the set partition a \newword{noncrossing set partition}. We use $\nc(n,d,r)$ to denote the set of a noncrossing set partitions of $n$ with $d$ blocks and blocks size at least $r$.

Let $\ordsetpart(n,d,r)$ denote the set of ordered set partitions of $n$ with $d$ blocks and blocks of size at least $r$. For example, $(2~5~6\mid 3 \mid 1~4)$ and $(3\mid 2~5~6 \mid 1~4)$ are distinct ordered set partitions in $\ordsetpart(6,3,1)$, but they are not distinct when viewed as unordered set partitions. Note that if $r \leq r'$, then $\ordsetpart(n,d,r') \subseteq \ordsetpart(n,d,r)$. We write $\ncop(n,d,r)$ for the noncrossing subset of $\ordsetpart(n,d,r)$.

\begin{figure}[ht]
\includegraphics[scale=.9]{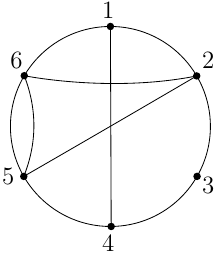}
\caption{A visual depiction of the ordered set partition $(2~5~6\mid 3 \mid 1~4)$ considered in Section~\ref{sec:set_partitions}. In this picture, the ordering of the blocks is not recorded. In later examples, we will occasionally use color-coding to illustrate the intended ordering, as needed.}
\label{fig:setpartition}
\end{figure}

\subsection{Grassmannians and exterior algebras}\label{sec:Grassmann}
The \newword{Grassmannian} $\Gr(k,n)$ is the parameter space of $k$-dimensional complex vector subspaces of $\C^n$. (In this paper, we will mostly be interested in the case of $\Gr(n,2n)$.) We realize $\Gr(k,n)$ as a smooth projective variety as follows. Consider a $k \times n$ matrix of distinct indeterminates 
\[
M = \begin{bmatrix}
    x_{11} & x_{12} & \dots & x_{1n}  \\
    x_{21} & x_{22} & \dots & x_{2n} \\
    \vdots & \vdots & \ddots & \vdots \\
    x_{k1} & x_{k2} & \dots & x_{kn}
\end{bmatrix}
\]
and let $\C[M]$ be the polynomial ring in these $kn$ variables. Let $R$ be the subring generated by all of the $k \times k$ minors of $M$. Then $R$ is the homogeneous coordinate ring of $\Gr(k,n)$, meaning that we can take $\Gr(k,n) \coloneqq \Proj R$. 

The action of $\SL_k(\C)$ on $M$ by left multiplication induces an action of $\C[M]$. By the Fundamental Theorem of Invariant Theory, the subring $R$ consists exactly of those polynomials in $\C[M]$ that are invariant under this $\SL_k(\C)$ action.

It is often convenient to replace the ring~$R$ by an isomorphic ring~$R'$. Here, $R' = \C[x_{\Delta} : \Delta \in \binom{[n]}{k}] / {\sim}$, where $\binom{[n]}{k}$ denotes the collection of $k$-element subsets of $\{1, 2, \dots, n\}$, each $x_\Delta$ stands in for the corresponding degree $k$ generator of $R$, and ${\sim}$ is the ideal corresponding to the relations among the generating minors of $R$. In this language, the indeterminates $x_\Delta$ are called \newword{Pl\"ucker variables} and elements of ${\sim}$ are called \newword{Pl\"ucker relations}. We will treat these two perspectives interchangeably. 

\subsection{Specht modules}
\label{sec:Spechtbackground}

Consider a $\Z^n$-grading on $\C[M]$, where each variable $x_{ij}$ has degree $\mathbf{e}_j$, where $\mathbf{e}_j$ is the $j$th standard basis vector of $\Z^n$. We are especially interested in the vector subspace $S$ of $R$ spanned by polynomials of multihomogeneous degree $(1,1, \dots, 1) = \sum_{j = 1}^n \mathbf{e}_j$. Note that $S$ consists exactly of multilinear $\SL_k$-invariant functions of the columns of $M$. We will assume that $n = dk$ for some integer $d$, so that $S$ is nontrivial.

It is well-known that $S$ is a finite-dimensional complex vector space whose dimension is given by the \emph{hook-length formula} for \emph{standard Young tableaux} of shape $(d^k)$. Precisely, 
\[
\dim S = \prod_{i=1}^{k} \frac{(d+k-i)!}{(k-i)!}.
\]

The symmetric group $\mathfrak{S}_n$ acts on $M$ by permuting columns, and hence on $\C[M]$. Note that $S$ is an invariant subspace for this action. In fact, it is irreducible as an $\mathfrak{S}_n$-module and is the \newword{Specht module} $S^{(d^k)}$.

In general, there is a Specht module $S^\lambda$ for every integer partition $\lambda = (\lambda_1 \geq \lambda_2 \geq \dots 0)$ with $\sum_i \lambda_i = n$; moreover, every irreducible complex representation of $\mathfrak{S}_n$ is isomorphic to exactly one of these Specht modules. Here, we give an explicit but terse construction of general Specht modules along the lines of the construction of $S^{(d^k)}$ above; for more details of this approach, see \cite[$\mathsection 2.3$]{Patrias.Pechenik.Striker}, while for more standard (and more thorough) textbook treatments, see, e.g., \cite{Fulton.Harris,Fulton:YoungTableaux,Sagan}. 

Let $\lambda$ be any partition of $n$ and let $\mu$ be the partition whose Young diagram is the transpose of that of $\lambda$. For example, if $\lambda = (d^r,1^{n - r \cdot d})$, then $\mu  = (n - (r-1) \cdot d, r^{d-1})$. Suppose that 
\[
\mu = (\mu_1 \geq \mu_2 \geq \dots \geq \mu_\ell > 0).
\]
Let $M$ be the matrix of indeterminates from Section~\ref{sec:Grassmann} with $k = \mu_1$. Consider the polynomial ring $\mathbb{C}[M]$ in those $n \cdot \mu_1$ variables. For each $\mu_i$, let $\mathsf{S}_i$ denote the set of sequences $\mathbf{j} = (j_1, j_2, \dots, j_{\mu_i})$ with $1 \leq j_1 < j_2 < \dots < j_{\mu_i}$. (Note that if $\mu_i = \mu_{i'}$, then $\mathsf{S}_i = \mathsf{S}_{i'}$.) For $\mathbf{j} \in \mathsf{S}_i$, let $p_{\mathbf{j}} \in \mathbb{C}[M]$ denote the polynomial that is the $\mu_i \times \mu_i$ minor of $M$ involving the top $\mu_i$ rows and the columns indexed by the sequence $\mathbf{j}$. Now, consider the set $\mathsf{S}$ of tuples of sequences $(\mathbf{j}_1, \dots, \mathbf{j}_\ell)$ such that each $\mathbf{j}_i \in \mathsf{S}_i$ and the sets $\mathbf{j}_i$ partition $[n]$ (that is, they are pairwise disjoint with union $[n]$).
For each $(\mathbf{j}_1, \dots, \mathbf{j}_\ell) \in \mathsf{S}$, let \[
p_{(\mathbf{j}_1, \dots, \mathbf{j}_\ell)} \coloneqq \prod_{i=1}^\ell p_{\mathbf{j}_i}.
\]
The vector span of the polynomials $p_{(\mathbf{j}_1, \dots, \mathbf{j}_\ell)}$ for $(\mathbf{j}_1, \dots, \mathbf{j}_\ell) \in \mathsf{S}$ is the Specht module $S^\lambda$. Here, the symmetric group $\mathfrak{S}_n$ is acting by permuting the columns of $M$.

The construction above naturally identifies $S^\lambda$ with a slice of the multihomogeneous coordinate ring of a partial flag variety, extending the construction above of $S^{(d^k)}$ inside the homogeneous coordinate ring of a Grassmannian; for more details, see \cite[$\mathsection 2.3$]{Patrias.Pechenik.Striker}.
One of our main results in Section~\ref{sec:GrassmannCayley} will be a new, and somewhat mysterious, construction of the Specht module $S^{(d^r,1^{n-r\cdot d})}$ inside the homogeneous coordinate ring of the Grassmannian $\Gr(n,2n)$.

\section{Jellyfish tableaux and  invariants}
\label{sec:invariants}
In this section, we define a polynomial associated to each ordered set partition in $\ordsetpart(n,d,r)$; these polynomials will be our web invariants. Although an ordered set partition $\pi \in \ordsetpart(n,d,r)$ is also an element of $\ordsetpart(n,d,r')$ for any $r' \leq r$, the associated web invariant depends on the specified $r$. When $r$ is even, the invariants we define will be independent of the ordering of the blocks; that is to say, they depend only on the underlying unordered set partition. When $r$ is odd, the ordering of the blocks changes the invariant by a predictable global sign, which we discuss in Lemma~\ref{lem:row_col_swap}. For the case $r=2$, our invariants recover those introduced in \cite{Patrias.Pechenik.Striker}.

We begin by defining a set of tableaux that we will use to construct polynomials associated to ordered set partitions. Because the shape consists of $r$ full rows of boxes followed by one box in each additional row and thus resembling jellyfish, we call these $r$-\emph{jellyfish tableaux} (cf.\ Example~\ref{ex:RStableaux}).

\begin{figure}[ht]
\raisebox{1.4in}{\begin{ytableau}
1 & 3 & 5 & 2\\
4 & 8 & 6 & 7\\
12 & 10 & 9 & 11\\
\none & 16 & \none \\
\none & \none & 13 & \none  \\
15 & \none & \none\\
\none & \none & \none & 14
\end{ytableau}}\hspace{.65in}
\includegraphics[scale=0.25]{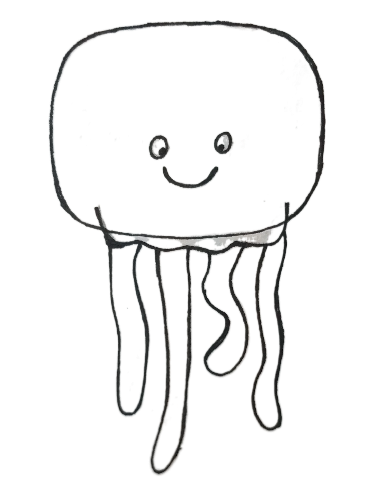}
\caption{A jellyfish tableau and a cute happy jellyfish.}\label{fig:cute_jellyfish}
\end{figure}

\begin{definition}
\label{def:rstab}
Given an ordered set partition $\pi=(\pi_1\mid \pi_2\mid \ldots\mid \pi_d)\in\ordsetpart(n,d,r)$, let $\rstab_r(\pi)$ be the set of generalized tableaux $T_{ij}$ (in English notation with matrix indexing) with $d$ columns (so $1\leq j\leq d$) and $n-(d-1)r$ rows ($1\leq i\leq n-(d-1)r$) obeying the following constraints:
\begin{enumerate}
\item $T_{ij}\in[n]$ or $T_{ij}$ is empty.
\item If $i\in[r]$, $T_{ij}$ is nonempty.
\item If $i>r$, there exists exactly one $j$ such that $T_{ij}$ is nonempty.
\item The nonempty entries in column $j$ are exactly the elements of $\pi_j$, in increasing order.
\end{enumerate}
Call $\rstab_r(\pi)$ the set of \newword{$r$-jellyfish tableaux} for $\pi$.
\end{definition}

\begin{definition}
\label{def:inv}
For $w \in \mathfrak{S}_n$, the \newword{inversion number} $\inv(w)$ is the number of \newword{inversions} in $w$, i.e., pairs $i<j$ such that $w(i) > w(j)$. The \newword{sign} of $w$ is $\sgn(w)=(-1)^{\inv(w)}$.

For $T\in \rstab_r(\pi)$, we define its \newword{inversion number} $\inv(T)$ to be the number of inversions in its row reading word (left to right, top to bottom) and its \newword{sign} to be the sign of its row reading word.
\end{definition}

Note that $i < j$ form an inversion of $T$ if and only if either $j$ appears in a higher row than $i$ or else $j$ appears left of $i$ in the same row.

The following lemma describes how the signs of $r$-jellyfish tableaux change under a permutation of the blocks of the corresponding ordered set partition. Note the sign is preserved when $r$ is even.

\begin{lemma}\label{lem:colswap}
Given an ordered set partition $\pi=(\pi_1\mid \pi_2\mid \ldots\mid \pi_d)\in\ordsetpart(n,d,r)$, a permutation $\sigma\in \mathfrak{S}_d$ induces a bijection between $\rstab_r(\pi)$ and $\rstab_r(\sigma(\pi))$, where $\sigma(\pi)_i = \pi_{\sigma^{-1}(i)}$. In this bijection,
$T\in\rstab_r(\pi)$ maps to $T'\in\rstab_r(\sigma(\pi))$ by permuting columns according to $\sigma$. Then \[\sgn(T)=\sgn(\sigma)^r\sgn(T').\]
\end{lemma}
\begin{proof}
It suffices to prove the lemma for $\sigma = s_i$ a simple transposition, swapping $i$ and $i+1$.
Suppose $\pi=(\pi_1\mid \pi_2\mid \ldots\mid \pi_d)\in\ordsetpart(n,d,r)$ and suppose $\pi'$ is obtained from $\pi$ by swapping $\pi_i$ and $\pi_{i+1}$. 

For each $T\in\rstab_r(\pi)$, there is a corresponding tableau $T'\in\rstab_r(\pi')$ given by swapping columns $i$ and $i+1$ of $T$. It remains to compare the inversion counts between $T$ and $T'$. Inversions in $T$ between entries in different rows are the same as inversions in $T'$ between entries in different rows. 
Moreover, inversions involving a column other than $i$ or $i+1$ are the same in $T$ and $T'$.
Thus, we need only consider inversions between entries in columns $i$ and $i+1$ that are in a single row. The only rows that have more than one entry are rows $1,2,\ldots,r$. If there were an inversion between the pair of entries in cells $(s,i)$ and $(s,i+1)$ of $T$,  where $1\leq s \leq r$, there would not be an inversion between the pair of entries in cells $(s,i)$ and $(s,i+1)$ of $T'$. If there were not an inversion between the pair of entries in cells $(s,i)$ and $(s,i+1)$ of $T$, there would be an inversion between the pair of entries in cells $(s,i)$ and $(s,i+1)$ of $T'$. This does not change the parity of the number of inversions when $r$ is even and does change the parity when $r$ is odd. Thus each simple transposition of columns of the tableau changes the sign of the tableau: ${\sgn(T)}=(-1)^r{\sgn(T')}$. The lemma follows.
\end{proof}

\begin{example}
Consider the tableaux $U$, $T$, $U'$, and $T'$ shown below. 
Let $\pi = (2\ 3\ 6\ 10 \mid 5\ 7\ 8\ 9 \mid 1\ 4)$ and let $\pi' = (5\ 7\ 8\ 9\ \mid 1\ 4\ \mid 2\ 3\ 6\ 10) = \sigma(\pi)$, where $\sigma$ is the permutation $312$.
We see that $U\in\rstab_2(\pi)$ and $T\in\rstab_2(\pi')$. Let $\rho = (2\ 3\ 6\ 7\ 12 \mid 1\ 8\ 10 \mid 4\ 5\ 9\ 11)$ 
and let $\rho' = (4\ 5\ 9\ 11 \mid 1\ 8\ 10 \mid 2\ 3\ 6\ 7\ 12)  = \tau(\rho)$, where $\tau$ is the permutation $321$. 
Then $U' \in \rstab_3(\rho)$ and $T' \in \rstab_3(\rho')$. 
Note that $\sgn(\sigma) = (-1)^2 = 1$, while $\sgn(\tau) = (-1)^3 = -1$.

The pairs of entries $(4,7)$, $(3,5)$, and $(9,10)$, for example, form inversions in $T$. The reader may check that $\inv(T)=9$ and $\inv(U)=7$. We may also find that $\inv(U')=9$ and $\inv(T')=12$. While $\inv(U)\neq\inv(T)$, we see that $\sgn(U)=\sgn(T)$, as predicted by Lemma~\ref{lem:colswap}.
On the other hand, $\sgn(U')\neq\sgn(T')$, which is consistent with Lemma~\ref{lem:colswap} since $r=3$ in this case and $\sgn(\tau) = -1$.

\[U=\begin{ytableau}
2 & 5 & 1 \\
3 & 7 & 4 \\
6 \\
\none & 8\\
10 \\
\none & 9
\end{ytableau}\hspace{.5in}
T=\begin{ytableau}
5 & 1 & 2 \\
7 & 4 & 3\\
\none & \none & 6\\
 8\\
\none & \none & 10\\
9
\end{ytableau}\hspace{.5in}
U'= \begin{ytableau}
2 & 1 & 4\\
3 & 8 & 5 \\
6 & 10 & 9\\
7 & \none \\
12 & \none & \none \\
\none & \none & 11
\end{ytableau}\hspace{.5in}
T'=\begin{ytableau}
4 & 1 & 2\\
5 & 8 & 3 \\
9 & 10 & 6\\
\none & \none & 7 \\
\none & \none & 12 \\
11 & \none & \none
\end{ytableau}\]
\end{example}

Recall the matrix $M$ from Section~\ref{sec:Grassmann}. Let $I$ and $J$ be finite subsets of $\mathbb{N}$ and let $M_I^J$ denote the determinant of the submatrix of $M$ with rows indexed by $I$ and columns indexed by $J$, in increasing order. (For convenience, we sometimes write elements separated by commas in the subscript and superscript rather than formal sets.) Note since these minors may not be top-justified, they are not naturally expressed in Pl\"ucker variables.

\begin{definition}
\label{def:rsdet}
Given $\pi=(\pi_1\mid \pi_2\mid\ldots\mid \pi_d)\in\ordsetpart(n,d,r)$ and $T\in \rstab_r(\pi)$, define the product of determinants \[\rspoly(T)= 
\displaystyle\prod_{i=1}^d M_{R_i(T)}^{\pi_i},\] where $R_i(T)$ is the set of rows of $T$ containing an entry in $\pi_i$.
\end{definition}

\begin{example}\label{ex:3-jellyfish}
For example, below is a $3$-jellyfish tableau and its corresponding product of determinants. 
\begin{center}
\raisebox{4ex}{$T= \begin{ytableau}
2 & 1 & 4\\
3 & 8 & 5 \\
6 & 10 & 9\\
7 & \none \\
\none & \none & 11 \\
12 & \none & \none
\end{ytableau}$}
\hspace{.5in}
$\rspoly(T)=\begin{vmatrix}
x_{12} & x_{13} & x_{16} & x_{17} & x_{1,12}\\
x_{22} & x_{23} & x_{26} & x_{27} & x_{2,12}\\
x_{32} & x_{33} & x_{36} & x_{37} & x_{3,12}\\
x_{42} & x_{43} & x_{46} & x_{47} & x_{4,12}\\
x_{62} & x_{63} & x_{66} & x_{67} & x_{6,12}
\end{vmatrix}\cdot
\begin{vmatrix}
x_{11} & x_{18} & x_{1,10}\\
x_{21} & x_{28} & x_{2,10}\\
x_{31} & x_{38} & x_{3,10} 
\end{vmatrix}\cdot
\begin{vmatrix}
x_{14} & x_{15} & x_{19} & x_{1,11} \\
x_{24} & x_{25} & x_{29} & x_{2,11} \\
x_{34} & x_{35} & x_{39} & x_{3,11} \\
x_{54} & x_{55} & x_{59} & x_{5,11} 

\end{vmatrix}
$

\end{center}
\end{example}

Note that Definition~\ref{def:rsdet} reduces to the Pl\"ucker case when $\pi$ is an ordered set partition with equally sized blocks.

We now define a polynomial invariant for each ordered set partition $\pi\in\ordsetpart(n,d,r)$. 

\begin{definition}
\label{def:polynomial}
Given an ordered set partition $\pi\in\ordsetpart(n,d,r)$, let $[\pi]_r$ denote the \newword{$r$-jellyfish invariant} \[
[\pi]_r = \sum_{T\in \rstab_r(\pi)}\sgn(T) \; \rspoly(T).
\]
If $\theta$ is an ordered set partition with a block of size less than $r$, we set $[\theta]_r=0$.
\end{definition}

\begin{remark}
In \cite{Patrias.Pechenik.Striker}, $2$-jellyfish invariants for noncrossing set partitions were called \emph{web invariants}. In Section~\ref{sec:GrassmannCayley}, we explain how to realize arbitrary $r$-jellyfish invariants as classical tensor invariants; however, the associated tensor diagrams are highly nonplanar, even when the set partition is noncrossing.
\end{remark}

\begin{example}\label{ex:RStableaux} 
Suppose $\pi=(2~3~6~10 \mid 5~7~8~9 \mid 1~4)$. Then $[\pi]_r=0$ for $r>2$. To compute $[\pi]_2$, we first see that $\rstab_2(\pi)$ consists of the $2$-jellyfish tableaux below.
\[\begin{ytableau}
2 & 5 & 1 \\
3 & 7 & 4 \\
6 \\
10\\
\none & 8 \\
\none & 9
\end{ytableau}\hspace{.3in}
\begin{ytableau}
2 & 5 & 1 \\
3 & 7 & 4 \\
6 \\
\none & 8\\
10 \\
\none & 9
\end{ytableau}\hspace{.3in}
\begin{ytableau}
2 & 5 & 1 \\
3 & 7 & 4 \\
6 \\
\none & 8\\
\none & 9\\
10
\end{ytableau}\hspace{.3in}
\begin{ytableau}
2 & 5 & 1 \\
3 & 7 & 4 \\
\none & 8\\
6 \\
10 \\
\none & 9
\end{ytableau}\hspace{.3in}
\begin{ytableau}
2 & 5 & 1 \\
3 & 7 & 4 \\
\none & 8\\
6 \\
\none & 9\\
10
\end{ytableau}\hspace{.3in}
\begin{ytableau}
2 & 5 & 1 \\
3 & 7 & 4 \\
\none & 8\\
\none & 9\\
6 \\
10
\end{ytableau}\]
The leftmost tableau has row reading word $2~5~1~3~7~4~6~10~8~9$ and thus has 8 inversions. Reading the list of tableaux from left to right, the tableaux have 8, 7, 6, 8, 7, and 8 inversions, respectively. Finally, we have that 
\begin{align*}[\pi]_2 = &M_{1,2,3,4}^{2,3,6,10}\cdot M_{1,2,5,6}^{5,7,8,9}\cdot M_{1,2}^{1,4}
- M_{1,2,3,5}^{2,3,6,10}\cdot M_{1,2,4,6}^{5,7,8,9}\cdot M_{1,2}^{1,4}
+ M_{1,2,3,6}^{2,3,6,10}\cdot M_{1,2,4,5}^{5,7,8,9}\cdot M_{1,2}^{1,4}\\
+ &M_{1,2,4,5}^{2,3,6,10}\cdot M_{1,2,3,6}^{5,7,8,9}\cdot M_{1,2}^{1,4}
- M_{1,2,4,6}^{2,3,6,10}\cdot M_{1,2,3,5}^{5,7,8,9}\cdot M_{1,2}^{1,4}
+ M_{1,2,5,6}^{2,3,6,10}\cdot M_{1,2,3,4}^{5,7,8,9}\cdot M_{1,2}^{1,4}.
\end{align*}
To compute $[\pi]_1$, we would first find all $\binom{7}{3,3,1}=140$ jellyfish tableaux in $\rstab_1(\pi)$. We have listed four such tableaux below.

\[\begin{ytableau}
2 & 5 & 1 \\
3 & \none  \\
\none &7 & \none\\
\none & 8\\
\none & 9\\
6 \\
10\\
\none & \none & 4
\end{ytableau}\hspace{.5in}
\begin{ytableau}
2 & 5 & 1 \\
\none & 7  \\
3 & \none\\
\none & 8\\
\none & 9\\
6 \\
10\\
\none & \none & 4
\end{ytableau}\hspace{.5in}
\begin{ytableau}
2 & 5 & 1 \\
\none &7 & \none\\
\none & 8\\
\none & \none & 4\\
3 & \none  \\
\none & 9\\
6 \\
10
\end{ytableau}\hspace{.5in}
\begin{ytableau}
2 & 5 & 1 \\
\none & \none & 4\\
\none &7 & \none\\
3 & \none  \\
\none & 8\\
6 \\
10\\
\none & 9
\end{ytableau}\]
The associated polynomials are, respectively,
\begin{align*}
(-1)^{12} &M_{1,2,6,7}^{2,3,6,10}\cdot M_{1,3,4,5}^{5,7,8,9}\cdot M_{1,8}^{1,4},\\
(-1)^{13} &M_{1,3,6,7}^{2,3,6,10}\cdot M_{1,2,4,5}^{5,7,8,9}\cdot M_{1,8}^{1,4},\\
(-1)^{12} &M_{1,5,7,8}^{2,3,6,10}\cdot M_{1,2,3,6}^{5,7,8,9}\cdot M_{1,4}^{1,4}, \text{ and}\\
(-1)^{9} &M_{1,4,6,7}^{2,3,6,10}\cdot M_{1,3,5,8}^{5,7,8,9}\cdot M_{1,2}^{1,4}.
\end{align*}
\end{example}

\begin{example}\label{ex:3row}
Consider $\pi=(2~3~6~7~12\mid 1~8~10\mid 4~5~9~11)$. The set $\rstab_3(\pi)$ is shown below.
\begin{center}
 \begin{ytableau}
2 & 1 & 4\\
3 & 8 & 5 \\
6 & 10 & 9\\
7 & \none \\
12 & \none & \none \\
\none & \none & 11
\end{ytableau}\hspace{1in}
 \begin{ytableau}
2 & 1 & 4\\
3 & 8 & 5 \\
6 & 10 & 9\\
7 & \none \\
\none & \none & 11 \\
12 & \none & \none
\end{ytableau}\hspace{1in}
 \begin{ytableau}
2 & 1 & 4\\
3 & 8 & 5 \\
6 & 10 & 9\\
\none & \none & 11 \\
7 & \none & \none \\
12 & \none & \none
\end{ytableau}
\end{center}
From this, we compute that 
\begin{align*}[\pi]_3=(-1)^{9}M_{1,2,3,4,5}^{2,3,6,7,12}\cdot M_{1,2,3}^{1,8,10}\cdot M_{1,2,3,6}^{4,5,9,11}+(-1)^{8}M_{1,2,3,4,6}^{2,3,6,7,12}\cdot M_{1,2,3}^{1,8,10}\cdot M_{1,2,3,5}^{4,5,9,11}\\
+(-1)^{9}M_{1,2,3,5,6}^{2,3,6,7,12}\cdot M_{1,2,3}^{1,8,10}\cdot M_{1,2,3,4}^{4,5,9,11}.
\end{align*}

\end{example}

\section{Jellyfish invariants are tensor diagram invariants}\label{sec:GrassmannCayley}
In this section, we first recall background on Grassmann--Cayley algebras and tensor diagrams. We then use these ideas to realize jellyfish invariants as elements of the homogeneous coordinate ring of a Grassmannian in two steps: first we show that that jellyfish invariants can be computed via a certain expression in the Grassmann--Cayley algebra, and second we interpret this Grassmann--Cayley expression in terms of tensor diagrams. Finally, we establish that jellyfish invariants are elements of flamingo Specht modules $S^{(d^r,1^{n-r\cdot d})}$ in Theorem~\ref{thm:inspecht}.

\subsection{Jellyfish tableaux and Grassmannians}\label{sec:jellyfish_on_Gr}
With a fixed value of~$n$ in mind, let 
$M_0$ be the $n \times n$ diagonal matrix whose diagonal entries read $+1,-1,+1,$ starting from the upper left corner. Thus, for $n=2$, we have $M_0 = \begin{pmatrix}
1 & 0 \\ 0 & -1
\end{pmatrix}$
and, for $n=3$, we have
$M_0 = \begin{pmatrix}
1 &  0 & 0 \\ 0 & -1 & 0 \\ 0 & 0 & 1
\end{pmatrix}$.

Given an $n \times n$ matrix $M$, we have an associated $n \times 2n $ matrix 
\begin{equation}\label{eq:Phi}
\Phi(M) := \begin{pmatrix}
M_0 & {M} 
\end{pmatrix}
\end{equation}
obtained by concatenating $M_0$ with $M$.
Clearly, $\Phi(M)$ has rank $n$, so we may think of $\Phi(M)$ as representing a point in the Grassmannian $\Gr(n,2n)$ by taking the span of the rows. The map $M \mapsto \Phi(M)$ is an isomorphism between ${\rm Mat}_{n \times n}$ and the open Schubert cell defined by $\Delta_{[1,n]} \neq 0$ inside 
$\Gr(n,2n)$. We also denote this isomorphism of algebraic varieties by $\Phi$. We write $\Phi^*$ for the induced map on coordinate rings, that is, on polynomials in the matrix entries. This ring map is the key bridge which allows us to translate the constructions from the previous section to the setting of Grassmannians.

One computes that
\begin{equation}\label{eq:PlutoMinor}
(-1)^{|I|}\Delta_{([n]\setminus I) \cup (J+n)}(\Phi(M)) = M_I^J,
\end{equation}
which allows us to translate statements about matrix minors of arbitrary size to statements about Pl\"ucker coordinates on ${\rm Gr}(n,2n)$. The advantage of this translation is that we already have a notion of tensor diagrams and webs in place for functions on Grassmannians. This translation is known to experts; the map $\Phi$ (or rather, a slight variant of this map with different choices of signs) appears frequently in the total positivity literature, as it allows one to translate the classical notion of total positivity for ${\rm GL}_n$ as a special case of total positivity for the Grassmannian ${\rm Gr}(n,2n)$. However, we are unaware of any explicit description in the literature for this approach to constructing Specht modules.

\begin{example}
Recall the jellyfish tableau
\[
T= \begin{ytableau}
2 & 1 & 4\\
3 & 8 & 5 \\
6 & 10 & 9\\
7 & \none \\
\none & \none & 11 \\
12 & \none & \none
\end{ytableau}
\] from Example~\ref{ex:3-jellyfish}. We computed in that example that $J(T)= M_{1,2,3,4,6}^{2,3,6,7,12} \cdot M_{1,2,3}^{1,8,10}\cdot M_{1,2,3,5}^{4,5,9,11}$.
We have by Equation~\eqref{eq:PlutoMinor} that
\begin{align*}M_{1,2,3,4,6}^{2,3,6,7,12} &=(-1)^5\Delta_{5,7,8,9,10,11,12,2+12,3+12,6+12,7+12,12+12}(\Phi(M)), \\
M_{1,2,3}^{1,8,10} &= (-1)^3\Delta_{4,\dots,12,1+12,8+12,10+12}(\Phi(M)), \text{ and} \\ M_{1,2,3,5}^{4,5,9,11} &= (-1)^4\Delta_{4,6,\ldots,12,4+12,5+12,9+12,11+12}(\Phi(M)).\end{align*}
Thus, $J(T)$ can be realized as an element of the homogeneous coordinate ring of ${\rm Gr}(12,24)$, by multiplying these three signed Pl\"ucker coordinates.
\end{example}

\subsection{Background on Grassmann--Cayley algebras}\label{sec:GrassmanCayleyalgebras}
Let $V = \mathbb{C}^n$ with exterior powers $\bigwedge^a(V) $ and exterior algebra $\bigwedge(V) = \bigoplus_{a} \bigwedge^a(V)$. 

The wedge product is a map $\bigwedge^a(V) \otimes \bigwedge^b(V) \to \bigwedge^{a+b}(V)$. There is also a ``dual'' map of sorts $\cap \colon \bigwedge^{n-a}(V) \otimes \bigwedge^{n-b}(V) \to \bigwedge^{n-a-b}(V)$ defined by
\begin{align}\label{eq:meetoperation}
v_1 \wedge \cdots \wedge v_{n-a} \otimes w_1 \wedge \cdots \wedge w_{n-b} &\mapsto \\ \sum_{i_1 < \cdots < i_{b}} {\rm sign}(i_1,\dots,i_b,j_1,\dots,j_{n-a-b}) &\det\left(v_{i_1},\dots,v_{i_b},w_1,\dots,w_{n-b}\right)v_{j_1} \wedge \cdots \wedge v_{j_{n-a-b}}\nonumber
\end{align}
where for each choice of $i_1,\dots,i_b$ we define $j_1,\dots,j_{n-a-b}$ to be the elements of $[n-a] \setminus \{i_1,\dots,i_b\}$ written in ascending order. We refer to this operation as \newword{cap}.

The cap operation $\cap$ is associative. It is commutative up to a predictable sign. See \cite[Chapter~3]{Sturmfels} for these and other properties. The exterior algebra together with the $\cap$ operation is known as the \newword{Grassmann--Cayley algebra}, and the operations of $\wedge$ and $\cap$ are referred to as the join and meet operations in this context. For us, this is a succinct algebraic formalism for writing down complicated polynomials in Pl\"ucker coordinates. Note that \cite{Sturmfels} uses the symbols $\vee$ and $\wedge$ where we use $\wedge$ and $\cap$, respectively; we motivate our notation in the following remark.

\begin{remark} 
A nonzero tensor $x \in \bigwedge^a(V)$ is \newword{decomposable} if it can be written as a wedge of several vectors: $x = v_1 \wedge \cdots \wedge v_a$ for some $v_1,\dots,v_a$. (Recall that any element of $\bigwedge^a(V)$ is a linear combination of decomposable tensors.) A decomposable tensor $x$ determines a vector subspace $A_x \subset V$, namely the subspace $A_x := \{v \in V \colon v \wedge x = 0\}$. Moreover, $v_1,\dots,v_a$ are a basis for $A_x$ whenever $x = v_1 \wedge \cdots \wedge v_a$. If $x \in \bigwedge^a(V)$ and $y \in \bigwedge^b(V)$ are decomposable tensors with $x \wedge y \neq 0$, one has 
$$A_{x \wedge y} = A_x \oplus A_y.$$
Thus, $\wedge$ is an exterior-algebraic interpretation of the direct sum of vector subspaces. 

In the same way, $\cap$ is an exterior-algebraic interpretation of the intersection of vector subspaces. That is, if $x$ and $y$ are decomposable with $\codim(A_x \cap A_y) = \codim A_x + \codim A_y$, then $x \cap y$ is decomposable and 
\[A_{x \cap y} = A_x \cap A_y.\]
The decomposability is not obvious from the formula \eqref{eq:meetoperation}, but it is nonetheless true. 
\end{remark}

\begin{remark}\label{rmk:cookingwithGC}
Let $v_1,\dots,v_{2n} \in V$ be the columns of an $n \times 2n$ matrix representing a point in the Grassmannian ${\rm Gr}(n,2n)$. A standard recipe is to construct polynomial functions on ${\rm Gr}(n,2n)$ by repeatedly composing the $\wedge$ and $\cap$ operations until we eventually arrive at $\bigwedge^0(V) = \mathbb{C}$ or $\bigwedge^n(V) = \mathbb{C}$. Any function obtained this way can be expanded as an explicit monomial in Pl\"ucker coordinates by expanding the terms in \eqref{eq:meetoperation}.
\end{remark}

\subsection{Background on tensor diagrams}\label{sec:TDbackground}
Tensor diagrams are a pictorial formalism for encoding elements of the Grassmannian coordinate ring. Thus, one can also think of them as a formalism for encoding ${\rm SL}_n$-invariant polynomials or ${\rm SL}_n$ tensor invariants.

The tensor diagram formalism is very closely related to the Grassmann--Cayley approach. An accessible introduction to these ideas is given in the introduction of \cite{Fomin.Pylyavskyy} in the case of ${\rm SL}_3$, which discusses both the local and global approaches we describe below. The main reference for the general ${\rm SL}_n$ case is   \cite{Cautis.Kamnitzer.Morrison}. Our viewpoint here closely follows \cite{Fraser.Lam.Le}, which also handles this case.

\begin{definition}\label{defn:tensordiagram}
Consider a disk with $m$ points marked $1,\dots,m$ in clockwise order on its boundary. A \newword{tensor diagram} $W$ of type $(n,m)$ is a bipartite graph drawn in this disk with the property that every marked boundary point is a vertex of~$W$ and is colored black, all other vertices of $W$ reside in the interior of the disk and are colored either white or black. Moreover, the edges of $W$ are weighted by elements of $[n]$ such that the sum of edge weights around every interior vertex equals~$n$. (It will be convenient in some formulas to write edges of weight $0$, by which we mean that those edges do not exist.) We consider tensor diagrams up to boundary-preserving isomorphism.
\end{definition}

We do not require that tensor diagrams be planar graphs; we call a tensor diagram a \newword{web} if it is planar. By the \newword{unclasping} of a tensor diagram, we mean the graph obtained by replacing each boundary vertex of degree $\delta \geq 2$ by $\delta$ vertices of degree $1$ (leaving all other edges and vertices in the graph intact). We call a tensor diagram a \newword{tree} if its unclasping is a tree in the usual graph-theoretic sense.

A tensor diagram $W$ of type $(n,m)$ determines a \newword{tensor diagram invariant} 
\begin{equation}\label{eq:globaldefinition}
[W] \in \mathbb{C}[{\rm Gr}(n,m)].
\end{equation}

For a self-contained definition of $[W]$, we refer to \cite[Definition 4.1]{Fraser:2column}, which is based on \cite[Lemma 5.4]{Fraser.Lam.Le}. While the details of the definition will not concern us here, we can summarize the idea of the definition as follows:
\begin{itemize}
    \item By multilinearity, an ${\rm SL}_n$ invariant polynomial is determined by how it evaluates on tensor products of standard basis vectors $e_i \in V$.
    \item To any such tensor product ${\bf e} = e_{i_1} \otimes \cdots \otimes e_{i_m}$, one can associate a sign, namely the inversion number of the sequence of indices $i_1,\dots,i_m$. 
    \item The evaluation of the invariant $[W]$ on the tensor product ${\bf e}$ is this sign multiplied by a certain graph-theoretic count, namely the number of \emph{consistent labelings} (see \cite[Definition~3.1]{Fraser:2column}) of $W$ with boundary $(i_1,\dots,i_m)$.   
\end{itemize} 

We will refer to this as the \newword{global definition} of the tensor invariant $[W]$, to be contrasted with the local definition introduced shortly.

\begin{remark}
The references \cite{Fraser:2column, Fraser.Lam.Le} restricted attention to the class of planar diagrams, but we would like to use these definitions for nonplanar diagrams. This is a mild extension, as one can always convert a nonplanar diagram to a linear combination of planar diagrams using the crossing removal skein relation \cite[Corollary 6.2.3]{Cautis.Kamnitzer.Morrison} (see also \cite[Equation (34)]{Fraser.Pylyavskyy}), invoke \cite[Lemma 5.4]{Fraser.Lam.Le} on each of the resulting planar diagrams, and then undo the crossing removal skein relation to conclude that \cite[Lemma 5.4]{Fraser.Lam.Le} also holds for the nonplanar diagram.
\end{remark}

There is a second approach to defining the invariant $[W]$, which is the perspective favored in \cite{Cautis.Kamnitzer.Morrison} (see, e.g., Equation (1.1) and Section 3.1 therein). Roughly speaking, in this \newword{local} approach, every white vertex encodes an invocation of the $\wedge$ operation, whereas every black vertex encodes an invocation of the \newword{dual exterior product} operation $\psi^{s,t} \colon \bigwedge^{s+t}V \to \bigwedge^s V \otimes \bigwedge^tV$. The vertices on the boundary of the tensor diagram encode the input vectors $(v_1,\dots,v_m)$, and the tensor invariant $[W]$ evaluates on this input by composing the exterior and dual exterior product maps in the manner indicated by the diagram $W$, in a similar spirit to Remark~\ref{rmk:cookingwithGC}.

To develop the second approach rigorously, one modifies Definition~\ref{defn:tensordiagram} as follows: 
\begin{enumerate}
    \item edges are now oriented, with each boundary vertex a source;
    \item for every interior leaf vertex, the incident edge has weight $n$; 
    \item the sum of weights of incoming edges matches the sum of weights of outgoing edges at every interior vertex that is not a leaf; and
    \item the bipartiteness condition is droppped (but the vertices must still be bicolored).
\end{enumerate} 
The leaf vertices are called \newword{tags}. They reflect the identifications $\bigwedge^n \mathbb{C}^n \cong \mathbb{C}$. The resulting combinatorial object is called a \newword {tagged} tensor diagram in \cite{Fraser.Lam.Le}. 

The key technical result of \cite{Fraser.Lam.Le} is that the tagged tensor diagram formalism is equivalent to the global definition we gave above, in the sense that any tagged tensor diagram $\widehat{W}$ can be converted in a purely combinatorial manner to a tensor diagram $W$ in such a way that the invariants match up to sign (i.e., the function obtained by composing the exterior and dual exterior product maps as described by $\widehat{W}$ matches the global definition of the invariant $[W]$). 

In our proof of Proposition~\ref{prop:webs_are_tensors}, we will want to make this conversion precise. 
In an effort to be self-contained, we appeal to the following construction \cite[Remark 3.3]{Fraser.Lam.Le} for turning a tensor diagram $W$ into a tagged one $\widehat{W}$ whose corresponding invariant agrees up to sign. 

A \newword{perfect orientation} $\mathcal{O}$ of a bipartite graph is an orientation of its edges such that each white vertex has outdegree one and each black vertex has indegree one. Given a tensor diagram $W$ with a perfect orientation $\mathcal{O}$, we obtain a tagged tensor diagram 
by complementing the weights $a \mapsto n-a$ of all edges oriented from a white vertex to a black vertex. The resulting graph may have some boundary sinks, and one turns these into boundary sources by adding a tag to each such edge (see \cite[Eq.~(3.3)]{Fraser.Lam.Le}). Finally, one eliminates any oriented cycles by adding additional tags.

\begin{remark}\label{rmk:WLOG}
The $\cap$ operation can be expressed as a composition of $\wedge$ and dual exterior product maps. Namely one has 
\[(v_1 \wedge\cdots\wedge v_{n-a}) \cap (w_1 \wedge \cdots \wedge w_{n-b}) = \psi^{n-a-b,b}(v_1\wedge\dots\wedge v_{n-a}) \wedge (w_1 \wedge \cdots \wedge w_{n-b}).\]
Thus, any function that can be built out of iterating $\wedge$ and $\cap$ operations as in Remark~\ref{rmk:cookingwithGC} is a tagged tensor diagram invariant, hence coincides with an ordinary tensor diagram invariant up to sign. Such a diagram will always be a tree, typically a nonplanar one.
\end{remark}

\subsection{Jellyfish invariants as Grassmann--Cayley expressions}
We now interpret the jellyfish invariant construction $\pi \mapsto [\pi]_r$ in the language of Grassmann--Cayley algebras. 

With vectors $v_1,\dots,v_{2n}$ fixed, for $J \subseteq \{1,\ldots,2n\}$ we write $v_J \coloneqq v_{j_1} \wedge \cdots \wedge v_{j_b} \in \bigwedge^b V$, where $j_1,\dots,j_b$ are the elements of $J$ written in ascending order. 

We introduce some bookkeeping notation we will use in this section. Given $\pi \in \ordsetpart(n,d,r)$ we set $\nu_i = |\pi_i| - r$ and set $\nu = r+\nu_1+\cdots +\nu_d = n-(r-1)d$. Thus $\nu_i$ is the number of ``extra'' rows used in column $i$ when we calculate jellyfish tableaux for $\mathcal{J}_r(\pi)$. And $\nu$ is the total number of rows used, i.e.\ the length of the first column in the corresponding Specht module. 
We put 
$S = [r+1,\nu]$ and $E = [\nu+1,n]$. One can think of $S$ as the rows containing the tentacles of the jellyfish tableaux for $\mathcal{OP}(n,d,r)$ and $E$ as the rows strictly below all the tentacles.

The following definition is the main link between jellyfish invariants and the Grassmann--Cayley algebra. If $X$ is a set of numbers, we define $X+ n = \{ x + n : x \in X\}$.
\begin{definition}
\label{def:web_inv}
Let $\pi \in \ordsetpart(n,d,r)$. We obtain a function that evaluates on column vectors $v_1,\dots,v_{2n} \in V$ as follows: 
\begin{align}\label{eq:meetjellyfish}
v_1,\dots,v_{2n} \mapsto \left(\bigwedge_{i=1}^{d-1} v_{S} \cap v_{E \cup (\pi_i+n)} \right) \wedge v_{E \cup (\pi_d+n)}.
\end{align}
We denote this function by $[\pi]'_r$. 
\end{definition}

\begin{lemma}
   Let $\pi \in \ordsetpart(n,d,r)$. Then $[\pi]'_r$ is in $\mathbb{C}[{\rm Gr}(n,2n)]$.
\end{lemma}
\begin{proof}
The $i$th term in the wedge product of \eqref{eq:meetjellyfish} is $ v_{S} \cap v_{E \cup (\pi_i+n)} \in \bigwedge^{\nu_i}(V)$. Wedging these together for $i=1,\dots,d-1$ gives an element of $\bigwedge^{\nu_1+\cdots +\nu_{d-1}}(V)$. Further wedging with $v_{E \cup (\pi_i+n)}$ gives an element of 
$\bigwedge^{n}(V)$, which is $1$-dimensional and is identified with $\mathbb{C}$. Thus, we get a function on $(v_1,\dots,v_{2n})$. Moreover, by expanding each parenthesized factor as a signed sum using \eqref{eq:meetoperation}, we can express this function as a signed sum of monomials in Pl\"ucker coordinates. We get one Pl\"ucker coordinate from each of the parenthesized factors, and one more Pl\"ucker coordinate from the wedge with $v_{E \cup (\pi_i+n)}$, so the formula \eqref{eq:meetjellyfish} is a signed sum of degree~$d$ monomials in Pl\"ucker coordinates. 
\end{proof}

Our main result in this subsection is that that $r$-jellyfish invariants match the Grassmann--Cayley expression \eqref{eq:meetjellyfish} up to sign. Recall the map $\Phi$ from \eqref{eq:Phi}.
\begin{proposition}\label{prop:jellyfish_and_GC}
For an ordered set partition $\pi \in \ordsetpart(n,d,r)$, we have 
$$\Phi^*([\pi]'_r) = \pm [\pi]_r,$$
i.e., the $r$-jellyfish invariant coincides with the Grassmann--Cayley expression up to sign.
\end{proposition}

The proof of Proposition~\ref{prop:jellyfish_and_GC} yields a formula for this global sign, although it is somewhat complicated and perhaps not very useful. It is possible that this sign simplifies nicely, but we did not find such a simplification. An alternative way to handle the signs in practice is to expand both the left and right hand sides as polynomials in matrix minors and compare the sign of a matching pair of terms. Before proving Proposition~\ref{prop:jellyfish_and_GC}, we illustrate it with an example. 

\begin{example}\label{ex:cap}
In Example~\ref{ex:RStableaux}, recall $r=2$ and $\pi=(2~3~6~10 \mid 5~7~8~9 \mid 1~4)$. We have $\nu_1 = \nu_2 = 2$ and $\nu_3 = 0$, have $\nu = 6$, have $n=10$. Moreover, $S = \{3,4,5,6\}$ and $E = \{7,8,9,10\}$.
Hence, we have the function 
\[v_1,\dots,v_{20} \mapsto \left(v_{3,\dots,6} \cap v_{7,\dots,10, 12, 13, 16, 20} \right) \wedge \left(v_{3,\dots,6} \cap v_{7,\dots,10, 15, 17, 18, 19} \right) \wedge v_{7,\dots,10, 11, 14}.\]

We compute the first cap $v_{3,\dots,6} \cap v_{7,\dots,10, 12, 13, 16, 20}$ by choosing two out of the four vectors in $v_{3},\dots,v_{6}$ and moving them over to $v_7,\dots,v_{10},v_{12},v_{13},v_{16},v_{20}$ to form a determinant (picking up a sign in accordance with Equation~\eqref{eq:meetoperation}). Note, here $a=6$ and $b=2$, so $n -a-b = 10-6-2 = 2$.

The first cap is below, where in the summation, for each choice of $3 \leq i_1 < i_2 \leq 6$, we define $j_1,j_{2}$ to be the elements of $\{3,4,5,6\} \setminus \{i_1,i_2\}$ written in ascending order. 
\begin{align*}
v_{3,4,5,6} \cap v_{7,8,9,10,12,13,16,20} &= \\ \sum_{3\leq i_1 < i_2\leq 6} {\rm sign}(i_1,i_2,j_1,j_{2}) &\det\left(v_{i_1},v_{i_2}, v_{7}, v_{8}, v_{9}, v_{10},v_{12}, v_{13}, v_{16}, v_{20}\right)v_{j_1} \wedge v_{j_{2}}\nonumber \\
= {\rm sign}(3,4,5,6) &\det\left(v_{3},v_{4}, v_{7}, v_{8}, v_{9}, v_{10},v_{12}, v_{13}, v_{16}, v_{20}\right)v_{5} \wedge v_{6} \nonumber \\
+ {\rm sign}(3,5,4,6) &\det\left(v_{3},v_{5}, v_{7}, v_{8}, v_{9}, v_{10},v_{12}, v_{13}, v_{16}, v_{20}\right)v_{4} \wedge v_{6} \nonumber \\
+ {\rm sign}(3,6,4,5) &\det\left(v_{3},v_{6}, v_{7}, v_{8}, v_{9}, v_{10},v_{12}, v_{13}, v_{16}, v_{20}\right)v_{4} \wedge v_{5} \nonumber \\
+ {\rm sign}(4,5,3,6) &\det\left(v_{4},v_{5}, v_{7}, v_{8}, v_{9}, v_{10},v_{12}, v_{13}, v_{16}, v_{20}\right)v_{3} \wedge v_{6} \nonumber \\
+ {\rm sign}(4,6,3,5) &\det\left(v_{4},v_{6}, v_{7}, v_{8}, v_{9}, v_{10},v_{12}, v_{13}, v_{16}, v_{20}\right)v_{3} \wedge v_{5} \nonumber \\
+ {\rm sign}(5,6,3,4) &\det\left(v_{5},v_{6}, v_{7}, v_{8}, v_{9}, v_{10},v_{12}, v_{13}, v_{16}, v_{20}\right)v_{3} \wedge v_{4} \nonumber ,
\end{align*}

The second cap is, where again in the summation, for each choice of $i_1 < i_2$ we define $j_1,j_{2}$ to be the elements of $\{3,4,5,6\} \setminus \{i_1,i_2\}$ written in ascending order. 
\begin{align*}
v_{3,4,5,6} \cap v_{7,8,9,10,15,17,18,19} &= \\ \sum_{3\leq i_1 < i_2\leq 6} {\rm sign}(i_1,i_2,j_1,j_2) &\det\left(v_{i_1},v_{i_2}, v_{7}, v_{8}, v_{9}, v_{10}, v_{15}, v_{17}, v_{18}, v_{19}\right)v_{j_1} \wedge v_{j_{2}}\nonumber \\
= {\rm sign}(3,4,5,6) &\det\left(v_3,v_4, v_{7}, v_{8}, v_{9}, v_{10}, v_{15}, v_{17}, v_{18}, v_{19}\right)v_5 \wedge  v_6 \nonumber \\
+ {\rm sign}(3,5,4,6) &\det\left(v_{3},v_5, v_7, v_8, v_9, v_{10}, v_{15}, v_{17}, v_{18}, v_{19}\right)v_4 \wedge v_6 \nonumber \\
+ {\rm sign}(3,6,4,5) &\det\left(v_3,v_6, v_7, v_8, v_9, v_{10}, v_{15}, v_{17}, v_{18}, v_{19}\right)v_4 \wedge v_5 \nonumber \\
+ {\rm sign}(4,5,3,6) &\det\left(v_4,v_5,v_7, v_8, v_9, v_{10}, v_{15}, v_{17}, v_{18}, v_{19}\right)v_3 \wedge v_6 \nonumber \\
+ {\rm sign}(4,6,3,5) &\det\left(v_4,v_6,v_7, v_8, v_9, v_{10}, v_{15}, v_{17}, v_{18}, v_{19}\right)v_3 \wedge v_5 \nonumber \\
+ {\rm sign}(5,6,3,4) &\det\left(v_5,v_6,v_7, v_8, v_9, v_{10}, v_{15}, v_{17}, v_{18}, v_{19}\right) v_3 \wedge v_4 \nonumber
\end{align*}

Wedging these two caps together, dropping the terms that are zero, and noting that, for example, we have ${\rm sign}(4,6,3,5)={\rm sign}(3,5,4,6)$, we then obtain
{\small 
\begin{align*}
    (&v_{{3},{4},{5},{6}} \cap v_{7, 8, 9, 10, 12,13,16,{20}}) \wedge (v_{3,4,5,6} \cap v_{7,8,9,10, 15, 17, 18, 19})   \\
     &=\det\left(v_{3},v_{4}, v_{7}, v_{8}, v_{9}, v_{10}, v_{12}, v_{13}, v_{16}, v_{20}\right)  \det\left(v_{5},v_{6},v_7,v_8,v_9,v_{10},v_{15}, v_{17}, v_{18}, v_{19}\right) v_{5} \wedge v_{6} \wedge v_{3} \wedge v_{4}\nonumber \\
     &+ \det\left(v_{3},v_{5},v_{7}, v_{8}, v_{9}, v_{10}, v_{12}, v_{13}, v_{16}, v_{20}\right) \det\left(v_{4},v_{6},v_7,v_8,v_9,v_{10},v_{15}, v_{17}, v_{18}, v_{19}\right) v_{4} \wedge v_{6} \wedge v_{3} \wedge v_{5} \nonumber \\
     &+\det\left(v_{3},v_{6},v_{7}, v_{8}, v_{9}, v_{10}, v_{12}, v_{13}, v_{16}, v_{20}\right) \det\left(v_{4},v_{5},v_7,v_8,v_9,v_{10},v_{15}, v_{17}, v_{18}, v_{19}\right) v_{4} \wedge v_{5} \wedge v_{3} \wedge v_{6} \nonumber \\
     &+ \det\left(v_{4},v_{5},v_{7}, v_{8}, v_{9}, v_{10}, v_{12}, v_{13}, v_{16}, v_{20}\right) \det\left(v_{3},v_{6},v_7,v_8,v_9,v_{10},v_{15}, v_{17}, v_{18}, v_{19}\right) v_{3} \wedge v_{6} \wedge v_{4} \wedge v_{5} \nonumber \\
     &+ \det\left(v_{4},v_{6},v_{7}, v_{8}, v_{9}, v_{10}, v_{12}, v_{13}, v_{16}, v_{20}\right) \det\left(v_{3},v_{5},v_7,v_8,v_9,v_{10},v_{15}, v_{17}, v_{18}, v_{19}\right) v_{3} \wedge v_{5}\wedge v_{4} \wedge v_{6} \nonumber \\
     &+\det\left(v_{5},v_{6},v_{7}, v_{8}, v_{9}, v_{10}, v_{12}, v_{13}, v_{16}, v_{20}\right) \det\left(v_{3},v_{4},v_7,v_8,v_9,v_{10},v_{15}, v_{17}, v_{18}, v_{19}\right) v_{3} \wedge v_{4} \wedge v_{5} \wedge v_{6} \nonumber \\
     &= \Bigg( \det\left(v_{3},v_{4}, v_{7}, v_{8}, v_{9}, v_{10}, v_{12}, v_{13}, v_{16}, v_{20}\right)  \det\left(v_{5},v_{6},v_7,v_8,v_9,v_{10},v_{15}, v_{17}, v_{18}, v_{19}\right) \nonumber \\
     &- \det\left(v_{3},v_{5},v_{7}, v_{8}, v_{9}, v_{10}, v_{12}, v_{13}, v_{16}, v_{20}\right) \det\left(v_{4},v_{6},v_7,v_8,v_9,v_{10},v_{15}, v_{17}, v_{18}, v_{19}\right) \nonumber \\
     &+ \det\left(v_{3},v_{6},v_{7}, v_{8}, v_{9}, v_{10}, v_{12}, v_{13}, v_{16}, v_{20}\right) \det\left(v_{4},v_{5},v_7,v_8,v_9,v_{10},v_{15}, v_{17}, v_{18}, v_{19}\right) \nonumber \\
     &+ \det\left(v_{4},v_{5},v_{7}, v_{8}, v_{9}, v_{10}, v_{12}, v_{13}, v_{16}, v_{20}\right) \det\left(v_{3},v_{6},v_7,v_8,v_9,v_{10},v_{15}, v_{17}, v_{18}, v_{19}\right) \nonumber \\
     &- \det\left(v_{4},v_{6},v_{7}, v_{8}, v_{9}, v_{10}, v_{12}, v_{13}, v_{16}, v_{20}\right) \det\left(v_{3},v_{5},v_7,v_8,v_9,v_{10},v_{15}, v_{17}, v_{18}, v_{19}\right) \nonumber \\
     &+ \det\left(v_{5},v_{6},v_{7}, v_{8}, v_{9}, v_{10}, v_{12}, v_{13}, v_{16}, v_{20}\right) \det\left(v_{3},v_{4},v_7,v_8,v_9,v_{10},v_{15}, v_{17}, v_{18}, v_{19}\right) \Bigg)\cdot(v_{3} \wedge v_{4} \wedge v_{5} \wedge v_{6}). \nonumber
\end{align*}
}

Finally, we wedge with $v_{7,8,9,10,11,14}$ to get  a 10-fold wedge $v_{3,4,5,6,7,8,9,10,11,14}$, which is the determinant $\det\left(v_3,v_4,v_5, v_6, v_7, v_8, v_9, v_{10}, v_{11}, v_{14}\right).$

Thus we have 
\begin{align}
(v_{3,\dots,6} & \cap v_{7,\dots,10,12,13,16,20} ) \wedge 
\left(v_{3,\dots,6} \cap v_{7,\dots,10,15,17,18,19} \right) \wedge v_{7,\dots,10,11,14} \nonumber \\
 &= \det\left(v_3,v_4,v_5, v_6, v_7, v_8, v_9, v_{10},v_{11}, v_{14}\right) \cdot \nonumber \\
&\Bigg( \det\left(v_{3},v_{4}, v_{7}, v_{8}, v_{9}, v_{10}, v_{12}, v_{13}, v_{16}, v_{20}\right)  \det\left(v_{5},v_{6},v_7,v_8,v_9,v_{10},v_{15}, v_{17}, v_{18}, v_{19}\right) \nonumber \\
     &- \det\left(v_{3},v_{5},v_{7}, v_{8}, v_{9}, v_{10}, v_{12}, v_{13}, v_{16}, v_{20}\right) \det\left(v_{4},v_{6},v_7,v_8,v_9,v_{10},v_{15}, v_{17}, v_{18}, v_{19}\right) \nonumber \\
     &+ \det\left(v_{3},v_{6},v_{7}, v_{8}, v_{9}, v_{10}, v_{12}, v_{13}, v_{16}, v_{20}\right) \det\left(v_{4},v_{5},v_7,v_8,v_9,v_{10},v_{15}, v_{17}, v_{18}, v_{19}\right) \nonumber \\
     &+ \det\left(v_{4},v_{5},v_{7}, v_{8}, v_{9}, v_{10}, v_{12}, v_{13}, v_{16}, v_{20}\right) \det\left(v_{3},v_{6},v_7,v_8,v_9,v_{10},v_{15}, v_{17}, v_{18}, v_{19}\right) \nonumber \\
     &- \det\left(v_{4},v_{6},v_{7}, v_{8}, v_{9}, v_{10}, v_{12}, v_{13}, v_{16}, v_{20}\right) \det\left(v_{3},v_{5},v_7,v_8,v_9,v_{10},v_{15}, v_{17}, v_{18}, v_{19}\right) \nonumber \\
     &+ \det\left(v_{5},v_{6},v_{7}, v_{8}, v_{9}, v_{10}, v_{12}, v_{13}, v_{16}, v_{20}\right) \det\left(v_{3},v_{4},v_7,v_8,v_9,v_{10},v_{15}, v_{17}, v_{18}, v_{19}\right) \Bigg). \nonumber
\end{align}

Now replace the determinants with the appropriate $M^J_I$ terms using Equation (\ref{eq:PlutoMinor}) (noting that $|I| = 2$, so we can drop the global signs $(-1)^{|I|}$ in this example):

\begin{align*}M_{1,2}^{1,4}\Bigg( M_{1,2,5,6}^{2,3,6,10}\cdot M_{1,2,3,4}^{5,7,8,9}
&-M_{1,2,4,6}^{2,3,6,10}\cdot M_{1,2,3,5}^{5,7,8,9}
+M_{1,2,4,5}^{2,3,6,10}\cdot M_{1,2,3,6}^{5,7,8,9}
+M_{1,2,3,6}^{2,3,6,10}\cdot M_{1,2,4,5}^{5,7,8,9}\\ &-M_{1,2,3,5}^{2,3,6,10}\cdot M_{1,2,4,6}^{5,7,8,9}+M_{1,2,3,4}^{2,3,6,10}\cdot M_{1,2,5,6}^{5,7,8,9}\Bigg).\end{align*}
Note that these are exactly the terms appearing in Example~\ref{ex:RStableaux} (written in reverse order) and with the same signs as appear there.
\end{example}

\begin{proof}[Proof of Proposition~\ref{prop:jellyfish_and_GC}]
To compute the right-hand side of Equation~\eqref{eq:meetjellyfish}, we first move $\nu - |\pi_i| = \nu - r - \nu_i$ vectors from the left side of the $i$th $\cap$ (that is, $v_S$) to its right side (that is, $v_{E \cup (\pi_i + n)}$). For each term $\tau$ of the right side and for each $i=1,\dots,d-1$, let $S_{\tau,i} \subseteq S$ be the set of vectors that are \emph{not} moved into the determinant and let $\overline{S_{\tau,i}} = S \setminus S_{\tau,i}$ be the set of vectors that are moved. This turns $\tau$ into an $n$ by $n$ determinant and picks up the sign $(-1)^{\inv(\overline{S_{\tau,i}}, S_{\tau,i})}$ from the definition of $\cap$ in Equation~\eqref{eq:meetoperation}.
Next, we wedge together all of the vectors that are {\sl not} moved together with the $v_{E \cup (\pi_d+n)}$ term. The result is a degree $d$ monomial in Pl\"ucker coordinates together with a sign. The function \eqref{eq:meetjellyfish} is the signed sum of these monomials. 

Note that $|S_{\tau,i}| = \nu_i$ does not depend on $\tau$. We define $S_{\tau,d}$ to be $S \setminus \bigcup_{i=1}^{d-1}S_{\tau,i}$. In order to get a nonzero evaluation, we need the disjoint union $S_{\tau,1} \sqcup \cdots \sqcup S_{\tau, d}$ to  equal  $S$. In other words, we need the $S_{\tau,i}$ sets to be pairwise disjoint.

The term $\tau$ is then the product 
\begin{equation}
\pm \prod_{i=1}^d\Delta_{\overline{S_{\tau,i}} \cup E \cup (\pi_i+n)}
\end{equation}
of Pl\"ucker coordinates.
Applying $\Phi^*$ to this, by Equation~\eqref{eq:PlutoMinor}, we get 
\begin{equation}
\pm \prod_{i=1}^d M^{\pi_i}_{[r] \cup S_{\tau,i}}.
\end{equation}

Consider the jellyfish tableau $T(\tau) \in \mathcal{J}_r(\pi)$ that uses rows $S_{\tau,i} \cup [r]$ in column $i$. (Clearly, the correspondence between terms $\tau$ and jellyfish tableaux $T(\tau)$ is bijective.) The above monomial in matrix minors matches a the term in the definition of $[\pi_r]$ corresponding to $T(\tau)$. What remains to be checked is that the signs implicit in \eqref{eq:meetjellyfish} match up with the signs from the definition of $[\pi]_r$ (i.e., the sign of the reading word of~$T(\tau)$). For the purposes of this analysis, we say a quantity is \emph{global} if it depends only on~$\pi$, i.e. not on the specific jellyfish tableau~$T(\tau)$ encoded by the specific sets $S_{\tau,1},\dots,S_{\tau,d}$. For example, the $\nu_i$'s are global quantities. 

Now, fix $\tau$ and write $T = T(\tau)$. We first unpack the sign of the reading word of $T$. Let $\word(T)$ be the reading word of $T$. We let $\widehat{T}$ be the ``greedily top-justified left to right'' jellyfish tableau in $\mathcal{J}_r(\pi)$. In Example~\ref{ex:RStableaux}, $\widehat{T}$ is the first tableau, and thus $\word(\widehat{T}) = 2\ 5\ 1\ 3\ 7\ 4\ 6\ 10\ 8\ 9$. We define a permutation $\jp_T$ to be $\word(\widehat{T})^{-1} \circ \word(T)$; thus $\jp_T$ encodes the ``jellyfish pattern'' of $T$. 
In Example~\ref{ex:RStableaux}, the first $6$ symbols of $\jp_T$ are always $1,\dots,6$. (In general, the first $rd$ symbols will be $1, \dots, rd$.)
The last four symbols in our example are 
\[7,8,9,10 \text{ or } 7,9,8,10 \text{ or } 7,9,10,8 \text{ or } 9,7,8,10 \text{ or } 9,7,10,8 \text{ or } 9,10,7,8,\]
respectively, in the same order as the jellyfish tableaux of Example~\ref{ex:RStableaux}.

Letting $S_i = S_{\tau, i}$ be the row indices $>r$ that are used column $i$ of $T$ as above, we see that
\begin{equation}
\inv(\jp_T) = \# \{u<v \in [n] \colon u \in S_i, \, v \in S_j \text{ and } i>j \},
\end{equation}
which we can write as $\inv(S_1,\dots,S_d)$ (with the number of inversions of a set-tuple defined to be the number of inversions of the permutation obtained by listing the elements in each set in increasing order). Since $\word(\widehat{T})$ is a global quantity, it suffices to show that $\inv(\jp_T)$ differs from the exterior algebra sign we discuss below by a global sign. 

Secondly, we unpack the exterior algebra signs associated with the calculation of the term $\tau$ in $[\pi]'_r$. The first type of sign contribution to $\tau$ comes due to bringing the set $S_i$ to the back of the ordered set $S$; hence, this  first sign is \[
\prod_{i=1}^{d-1} (-1)^{\inv(\overline{S_i},S_i)}.
\]
The second sign is the sign  $(-1)^{\inv(S_1,\dots,S_{d-1})}$ coming from wedging together the unmoved vectors for $i=1,\dots,d-1$ in the wedge product. 

We have $\inv(S_1,\dots,S_{d}) = \inv(S_1,\dots,S_{d-1})+\inv(\overline{S_d}, S_{d})$. 
Therefore, the sign $(-1)^{\inv(S_1,\dots,S_{d})}$ computed from the jellyfish tableau formula differs from the sign computed from Equation~\eqref{eq:meetoperation} by
\[
(-1)^{\sum_{i=1}^d \inv(\overline{S_i}, S_{i})}.
\]
Hence, it suffices to show that this number is a global quantity. A pair $u,v$ with $u \in S_i$ and $v \in S_j$ appears exactly once in this sum (it appears in the term indexed by $i$ if $v<u$ and otherwise appears in the term indexed by $j$). The number of such pairs only depends on the block sizes and thus is is a global quantity. (Specifically, it is $\frac{1}{2} \sum_{i=1}^d \nu_i(\nu-|\pi_i|)$.) This completes the proof.

Note that we have in fact shown that the sign appearing in the proposition is \[(-1)^{\inv(\word(\widehat{T})) + \frac{1}{2} \sum_{i=1}^d \nu_i(\nu-|\pi_i|) }. \qedhere\]
\end{proof}

\subsection{Jellyfish invariants as tensor diagram invariants}
As explained in Remark~\ref{rmk:WLOG}, any Grassmann--Cayley expression obtained by iterated composition of $\wedge$ and $\cap$ operations can be expressed as a tensor diagram invariant. In this section, we give a combinatorial construction of the tensor diagram that underlies the expression \eqref{eq:meetjellyfish}.

\begin{definition}\label{def:tensor_diagram_from_pi}
Given an ordered set partition $\pi \in \ordsetpart(n,d,r)$, 
we associate a tensor diagram $W_{\pi,r}$ of type $(n,2n)$ using the following steps. Recall we have set $S = [r+1,\nu]$ and $E = [\nu+1,n]$.  
\begin{enumerate}
\item 
For $i \in [d]$, introduce a interior white vertex $w_i$ and draw edges from $w_i$ to every boundary vertex in $E$, as well as to every boundary vertex in $\pi_i + n$. All of these edges receive weight $1$.
\item For $i\in [d-1]$, introduce an interior white vertex $u_i$ and draw edges from $u_i$ to every boundary vertex in $S$. These edges also all receive weight $1$.
\item For $i \in [d-1]$, introduce an interior black vertex $b_i$. For $i \in [d-1]$, draw an edge between each $b_i$ and the corresponding $w_i$ and give this edge weight $\nu-|\pi_i|$. Also, draw an edge between each $b_i$ and the corresponding $u_i$, giving this edge weight $n-\nu+r =rd$. Finally, draw an edge between each $b_i$ and the vertex $w_d$, giving this edge weight $\nu_i$. 
\end{enumerate}
\end{definition}

\begin{figure}[htbp]
\begin{center}
\includegraphics[width=5in]{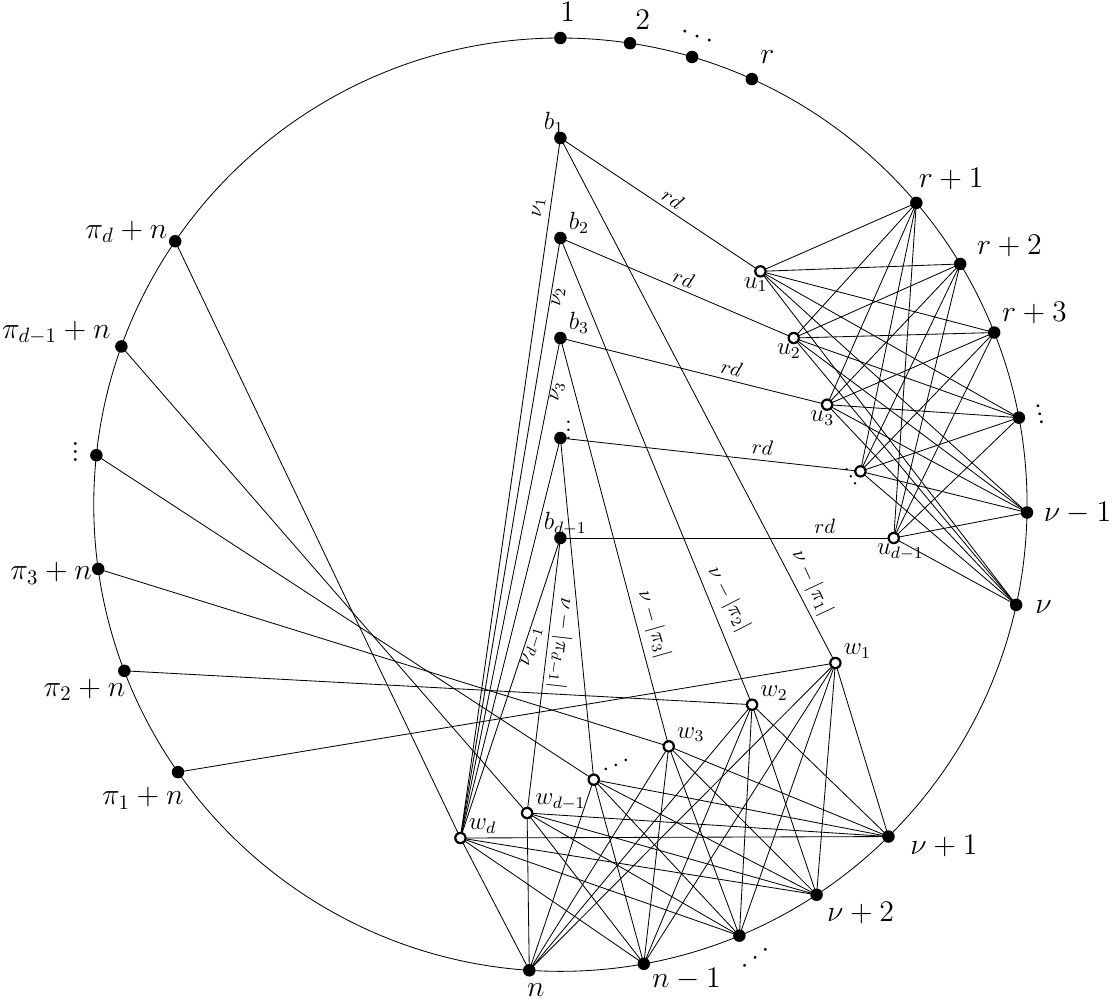}
\end{center}
\caption{A schematic tensor diagram $W_{\pi,r}$ illustrating the algorithm in Definition~\ref{def:tensor_diagram_from_pi}. Vertices $1,\dots,r$ are not used. Vertices in $S=\{r+1,\ldots,\nu\}$ and $E=\{\nu+1,\ldots,n\}$ are bookkeeping vertices such that each $u_i$ connects to all vertices in $S$ and each $w_i$ connects to all vertices in $E$. The vertices $\pi_i +n$ are schematic, and they encode the set partition $\pi$ in the sense that $w_i$ is connected to all of the $|\pi_i|$ boundary vertices in $\pi_i+n$. Note these vertices will not usually be cyclically ordered as we have drawn them here. See Figure~\ref{fig:running_tensor} for an explicit example.}
\label{fig:schematic}
\end{figure}

\begin{figure}[htbp]
    \centering
\includegraphics[width=5in]{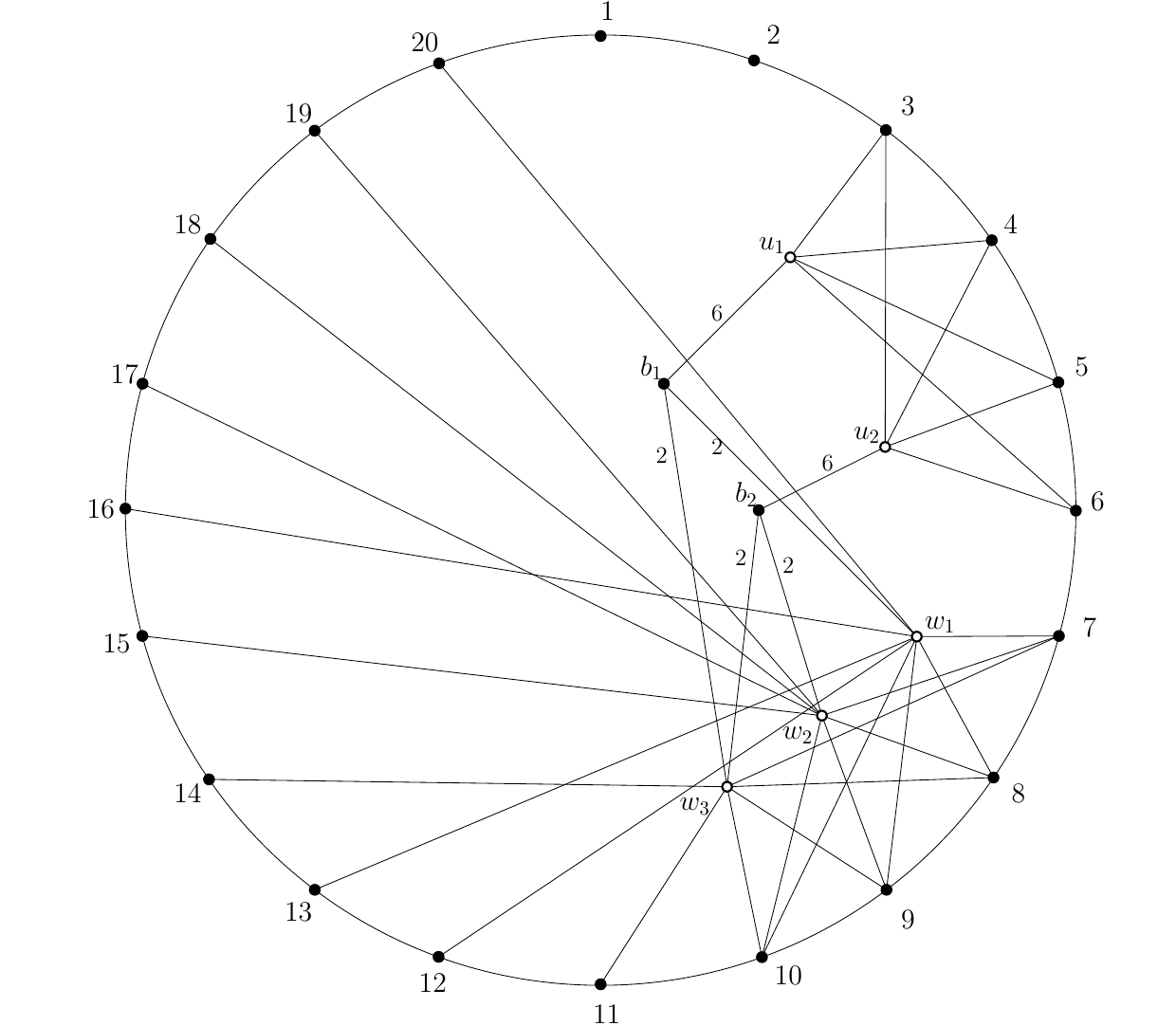}
    \caption{The tensor diagram obtained by Definition~\ref{def:tensor_diagram_from_pi} from ordered set partition $\pi=(2\ 3\ 6\ 10 \mid 5 \ 7\ 8\ 9 \mid 1\ 4)$ from Examples~\ref{ex:RStableaux} and \ref{ex:cap}. Note for example that $w_1$ connects to boundary vertices $12$, $13$, 16, and 20, which is $\pi_1+n$ in this example.}
    \label{fig:running_tensor}
\end{figure}

See Figure~\ref{fig:schematic} for a schematic diagram of this definition and Figure~\ref{fig:running_tensor} for an example. Clearly the tensor diagram $W_{\pi,r}$ is typically highly nonplanar. 

\begin{lemma}
    The diagram $W_{\pi,r}$ constructed in Definition~\ref{def:tensor_diagram_from_pi} is a tensor diagram of type $(n,2n)$.
\end{lemma}
\begin{proof}
    It is enough to verify that sum of weights around every internal vertex is $n$.

   For $i\in [d-1]$, each $w_i$  has 
\begin{itemize}
    \item $|\pi_i|$ edges of weight $1$ connecting $w_i$ to the boundary vertices in $\pi_i$,
    \item $n-\nu$ edges of weight $1$ connecting $w_i$ to the boundary vertices in $E$, and
    \item $1$ edge of weight $\nu-|\pi_i|$ connecting $w_i$ to $b_i$.
\end{itemize} 
Thus, the edges incident at $w_i$ have weights summing to \[n - \nu + \nu - |\pi_i| + |\pi_i| =n,\] as desired.
 
 The vertex $w_d$ has 
 \begin{itemize}
     \item $1$ edge of weight $\nu_i$ for each $i\in[d-1]$ connecting $w_d$ to $b_i$,
     \item $|\pi_d|=\nu_d+r$ edges of weight $1$ connecting $w_d$ to the boundary vertices in $\pi_d$, and
     \item $n-\nu = rd-r$ edges of weight $1$ connecting $w_i$ to the boundary vertices in $E$.
 \end{itemize}
 Thus, the edges incident at $w_d$ have weights summing to \[\Bigg( \sum_{i=1}^{d-1}\nu_i \Bigg) +\nu_d +r +rd-r = 
 rd + \sum_{i=1}^{d}\nu_i =n,\]
 as desired.

Each vertex $u_i$ has 
\begin{itemize}
    \item $1$ edge of weight $n-\nu+r$ connecting $u_i$ to the corresponding $b_i$ and
    \item $\nu-r$ edges of weight $1$ connecting $u_i$ to the boundary vertices in $S$,
\end{itemize}
giving a sum of $n$.

Finally, each vertex $b_i$ has 
\begin{itemize}
    \item $1$ edge of weight $n-\nu+r$ connecting $b_i$ and $u_i$,
    \item $1$ edge of weight $\nu-|\pi_i|=\nu-(r+\nu_i)$ connecting $b_i$ and $w_i$, and 
    \item $1$ edge of weight $\nu_i$ connecting $b_i$ and $w_d$.
\end{itemize}
Thus, the edges incident to $b_i$ have weight summing to 
\[
n-\nu + r + \nu-(r+\nu_i) +\nu_i= n,
\]
completing the proof.
\end{proof}

\begin{proposition}\label{prop:webs_are_tensors}
The tensor diagram $W_{\pi,r}$ encodes the Grassmann--Cayley expression \eqref{eq:meetjellyfish} up to sign.
\end{proposition}

\begin{proof}
The structure of the proof is as follows. First, we will interpret the Grassmann--Cayley expression \eqref{eq:meetoperation} as a sequence of $\wedge$ and dual exterior product operations, giving rise to a tagged tensor diagram $W_1$. By construction, this tagged tensor diagram $W_1$ computes the Grassmann--Cayley expression. Second, we will choose a tagging $W_2$ of $W_{\pi,r}$ by choosing a perfect orientation of the latter graph. Then we can apply \cite[Lemma 5.4]{Fraser.Lam.Le}, which asserts that the tagged tensor invariant $[W_2]$ agrees with $[W_{\pi,r}]$ up to sign.  Finally we will argue that $[W_1]$ and $[W_2]$ agree up to sign, completing the proof. 

For the first step, we express the Grassmann--Cayley expression in terms of tagged tensor diagram constructions. The tensor $v_{E\cup (\pi_i+n)}$ is diagrammatically encoded by an interior white vertex $w_i$ with incoming edges of weight 1 from the boundary vertices in $E \cup (\pi_i+n)$, for $i=1,\dots,d$ and with one outgoing edge of weight $|E|+|\pi_i|$. The tensor $v_S$ is encoded by a white vertex $u_i$ with incoming edges of weight 1 from the boundary vertices in $S$ and with one outgoing edge of weight $|S|$, for $i=1,\dots,d-1$. The cap $v_S \cap v_{E\cup (\pi_i+n)}$ of these tensors is encoded by a black vertex $b_i$ whose unique incoming edge emanates from vertex $u_i$ and which has two outgoing edges. The first of these outgoing edges has weight $n-|E| - |\pi_i| = \nu-|\pi_i|$ and meets the outgoing edge emanating from $w_i$ at a tag; this tag represents the explicit determinant in the definition of the meet operation \eqref{eq:meetoperation}. The second outgoing edge emanating from $b_i$ has weight $\nu_i$, representing the wedge after the determinant in  \eqref{eq:meetoperation}. It meets the edge emanating from $w_d$ at a tag, computing the wedge of the tensors just mentioned with the last factor $v_{E \cup (\pi_d+n)}$ in \eqref{eq:meetjellyfish}. Call the resulting tagged tensor diagram $W_1$. 

Second, we obtain a perfect orientation of $W_{\pi,r}$ as follows. Identify the sets $S$ and $E$ with their corresponding boundary vertices. We orient edges $S \to u_i$ for $i=1,\dots,d-1$; $u_i \to b_i$ for $i=1,\dots,d-1$; $b_i \to w_i$ and $b_i \to w_d$ for $i=1,\dots,d-1$. We orient edges $E \setminus \{n\} \to w_i$ and $w_i \to n$ for $i=1,\dots,d$. Applying the construction from Section~\ref{sec:TDbackground}, we obtain a tagged diagram $W_2$. 

To complete the proof, note that $W_1$ and $W_2$ differ only by applying tag migration moves \cite[Equation 2.7]{Cautis.Kamnitzer.Morrison}, migrating the tagged edge between boundary vertex $n$ and $w_i$ to the edges between the $b_i$'s and $w_i$'s. Tag migration moves only change sign, completing the proof.
\end{proof}

We now conclude the main result of this section.
That is, up to sign and application of the translation map $\Phi^*$, $r$-jellyfish invariants associated to ordered set partitions are tensor diagram invariants in the classical sense. 
\begin{theorem}
For any ordered set partition $\pi \in \ordsetpart(n,d,r)$, the corresponding $r$-jellyfish invariant is a tensor invariant; precisely, we have \[[\pi]_r = \pm \, \Phi^*([W_{\pi,r}]).\]
\end{theorem}
\begin{proof}
    Combining Propositions~\ref{prop:jellyfish_and_GC} and~\ref{prop:webs_are_tensors}, we obtain the desired result.
\end{proof}

\subsection{Jellyfish invariants lie in flamingo Specht modules}

We now use the commutativity of the $\cap$ operation to establish an important property of $r$-jellyfish invariants. The $r=2$ case is \cite[Lemma~3.14]{Patrias.Pechenik.Striker}.

\begin{theorem}\label{thm:inspecht}
For each ordered set partition $\pi \in \ordsetpart(n,d,r)$, the invariant $[\pi]_r$ lies in the flamingo Specht module $S^{(d^r, 1^{n-rd})}$.
\end{theorem}
\begin{proof}
Recall that $\cap$ is a commutative operation up to sign. Multiplicative factors of $-1$ do not affect our assertion. Therefore, up to global sign, we can compute the function $[\pi]'_r$ on $v_1,\dots,v_n$ by switching the two terms on each $\cap$ in the formula \eqref{eq:meetjellyfish} to obtain: 
\begin{equation}\label{eq:switch}
v_1,\dots,v_{2n} \mapsto \left(\bigwedge_{i=1}^{d-1} v_{E \cup (\pi_i+n)} \cap v_{S}\right) \wedge v_{ E\cup (\pi_d+n)}.
\end{equation}
Performing the computation in this way, we \emph{must} move all of the vectors indexed by elements of $E$ into the determinant with $v_{r+1},\ldots,v_{\nu}$ in order to get a nonzero evaluation when we later compute the wedge product with the final term $v_{E \cup (\pi_d+n)}$. 

For $i \in [d-1]$, let $\overline{Y_i}$  be the other vectors that are moved into the determinant and are indexed by elements in $\pi_i+n$. 
Let \[
Y_i = \{ v_k : k \in \pi_i+n, v_k \notin \overline{Y_i}\}
\]
and let 
\[Y_d = \{v_k : k \in \pi_d + n \} \cup \bigcup_{i=1}^{d-1}  Y_i.
\]

Then the right side of the above formula \eqref{eq:switch} expands as a linear combination 
\begin{equation}\label{eq:dets}
\sum_{\overline{Y_1},\dots,\overline{Y_{d-1}}}\pm \Delta_{E \cup Y_d}\prod_{i=1}^{d-1}\Delta_{  S \cup E\cup \overline{Y_i}},
\end{equation}
where the sum is over all possible sets $\overline{Y_i}$ of the appropriate size.

By \eqref{eq:PlutoMinor}, applying the translation map $\Phi^*$ to \eqref{eq:dets} yields a polynomial in matrix minors
\begin{equation}
\sum_{\overline{Y_1},\dots,\overline{Y_{d-1}}}\pm M^{Y_d}_{[\nu]}\prod_{i=1}^{d-1}M^{\overline{Y_i}}_{[r]}.
\end{equation}
Each monomial in this linear combination is manifestly an element of the flamingo Specht module $S^{(d^r, 1^{n-rd})}$, since it is a product of top-justified minors of the appropriate sizes (cf.\ Section~\ref{sec:Spechtbackground}). 
\end{proof}

We give an alternate proof of Theorem~\ref{thm:inspecht} in Section~\ref{sec:rest} by a recurrence that we establish in Theorem~\ref{thm:5term}.

\section{Diagrammatics of jellyfish invariants}\label{sec:diagrammatics}

In this section, we develop a diagrammatic calculus of jellyfish invariants. We also use this perspective to give an alternate proof of Theorem~\ref{thm:inspecht}. Subsection~\ref{sec:hook} focuses on the special ``hook'' case $r=1$.

\subsection{Inversion counting lemmas}
\label{sec:lemmas}

We now prove some useful lemmas on comparing the number of inversions in related tableaux. These will be used to understand signs in the proof of Theorem~\ref{thm:5term} in the next subsection.

\begin{definition}
Let $\rstabgen_r(T)$ be the set of tableaux obtainable from $T\in\rstab_r(\pi)$ by permuting entries within columns. Let $\rstabgen_r(\pi)=\displaystyle\bigcup_{T\in\rstab_r(\pi)}\rstabgen_r(T)$. We call the elements of $\rstabgen_r(\pi)$ \newword{column-permuted jellyfish tableaux}.
\end{definition}

\begin{definition}
\label{def:inv2}
Let $T\in \rstabgen_r(\pi)$. Define the \newword{inversion number} $\inv(T)$ as the number of inversions in the row reading word (left to right, top to bottom), except we do not count inversions within columns. Define $\sgn(T)=(-1)^{\inv(T)}$.
\end{definition}

Note that $\rstab_r(\pi) \subsetneq \rstabgen_r(\pi)$ and Definition~\ref{def:inv} agrees with the restriction of Definition~\ref{def:inv2}, since there are no inversions within columns of a tableau in $\rstab_r(\pi)$.

\begin{lemma}
\label{lem:row_col_swap}
Suppose $U\in \rstab_r(\pi)$ and $T\in \rstabgen_r(U)$. Then $\sgn(U)=\sgn(T)$.
\end{lemma}

\begin{proof} 
This argument is the same as for the analogous statement in \cite[Lemma~3.9]{Patrias.Pechenik.Striker}, and we will reproduce it here.
Suppose $T\in \rstabgen(U)$ and suppose $T'$ is obtained from $T$ by permuting the labels within columns.
We show $\sgn(T)=\sgn(T')$ (and hence, by repeated application, $\sgn(T) = \sgn(U)$). Note $\inv(T)$ may not equal $\inv(T')$, but they have the same parity.

It is enough to prove the result in the case that $T'$ is obtained from $T$ by interchanging two entries $i$ and $j$ in a single column.
Recall that $\inv(T)$ is the number of inversions in the row reading word of $T$, except ignoring inversions within columns. For each pair of entries $k,\ell$ of $T$, define
\[
\iota(k,\ell) = \begin{cases}
-1, & \text{if $k$ and $\ell$ are inverted in the reading word of $T$;} \\
1, & \text{otherwise.}\\
\end{cases}
\]

Consider any entry $k$ outside of the column containing $i$ and $j$.
Interchanging the elements $i$ and $j$ in $T$ swaps their
positions in the reading word. Therefore, the sum $\iota(i,k) + \iota(j,k) = 0$. Hence, this swap does not affect the parity of the number of inversions involving $(i, k)$ and $(j, k)$.
It does change whether $(i, j)$ is an inversion in the reading word; however, since they are in the same column, this does not affect the inversion number of the jellyfish tableau. Hence, $\sgn(T) = \sgn(T')$.
\end{proof}

\begin{corollary}
    Suppose that $\pi, \pi' \in \ordsetpart(n,d,r)$ are ordered set partitions with the same underlying unordered set partition, so that there exists $\sigma\in\mathfrak{S}_d$ with $\sigma(\pi)=\pi'$. Then, 
    \[
    [\pi]_r = \sgn(\sigma)^r [\pi']_r.
    \]
\end{corollary}
\begin{proof}
    This is immediate from combining Lemmas~\ref{lem:colswap} and~\ref{lem:row_col_swap}.
\end{proof}

\subsection{A \texorpdfstring{$(2^r+1)$}{2\^r+1}-term recurrence}
\label{sec:5term}
In this subsection, we state and prove a useful recurrence, Theorem~\ref{thm:5term}, which we apply in the next subsection to give an alternate proof of Theorem~\ref{thm:inspecht}, that the $r$-jellyfish invariants $[\pi]_r$ are contained in the appropriate Specht modules. The $r=2$ case of this recurrence appeared as \cite[Theorem~3.11]{Patrias.Pechenik.Striker} with a direct proof by Laplace expansion of determinants. Here, we identify the recurrence as an avatar of the classical Pl\"ucker relation. The description of signs appearing in Theorem~\ref{thm:5term} is very different from the classical Pl\"ucker sign calculation, so this perspective may be generally useful as a new way to think about the signs appearing in Pl\"ucker relations.

\begin{theorem}
\label{thm:5term}
Partition $\{1,\ldots n\}$ into three nonempty sets: $A$, $B$, $C$, where $|C|=r$. Then 
\begin{equation}
\label{eq:5term}
\big[(A\cup B \mid C)\big]_r=\sum_{S\subseteq C}(-1)^{|S|}\big[( A\cup S \mid B\cup (C\setminus S) ) \big]_r
\end{equation}
\end{theorem}

To illustrate Theorem~\ref{thm:5term}, let $r=3$ and suppose $C = \{c_1, c_2,c_3\}$. Figure~\ref{fig:recurrence} gives a diagrammatic representation of the $9$-term recurrence:
{\small
\begin{align*}
\big[&(A\cup B \mid C)\big]_3
=\\ 
&\big[(A \mid B \cup C)\big]_3
-\big[(A\cup \{c_1\}  \mid B \cup \{c_2,c_3\})\big]_3
-\big[(A\cup \{c_2\}  \mid B \cup \{c_1,c_3\})\big]_3
-\big[(A\cup \{c_3\}  \mid B \cup \{c_1,c_2\})\big]_3 \\
+&\big[(A\cup \{c_1,c_2\} \mid B  \cup \{c_3\})\big]_3
+\big[(A\cup \{c_1, c_3\}  \mid B \cup \{c_2\})\big]_3
+\big[(A\cup \{c_2,c_3\} \mid B \cup \{c_1\})\big]_3
-\big[(A\cup C \mid B)\big]_3
\end{align*}
}

\begin{figure}[htbp]
\begin{center}
\includegraphics[width=\textwidth]{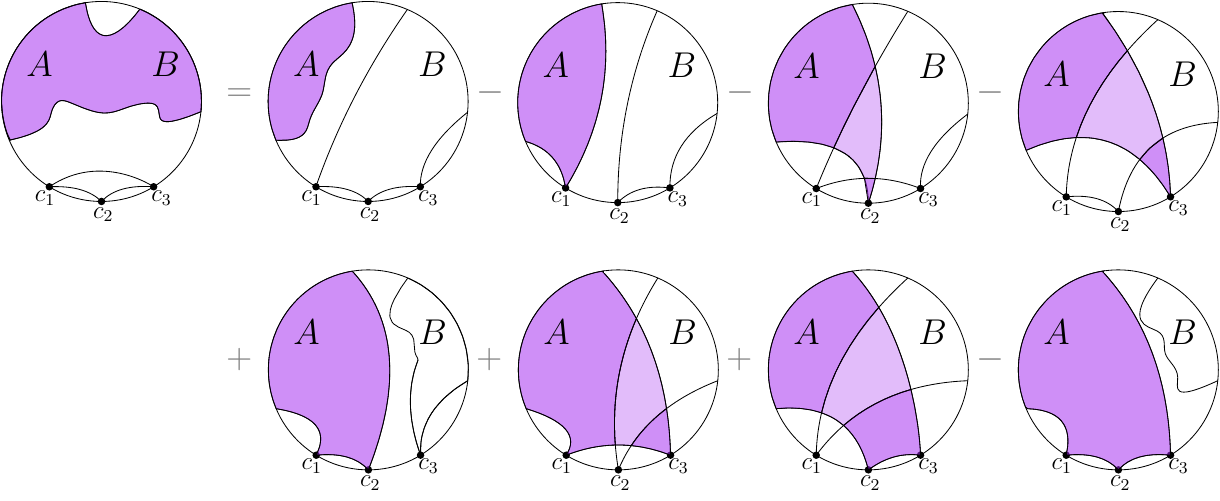}
\caption{An illustration of the recurrence of Theorem~\ref{thm:5term} in the case $r=3$.}
\label{fig:recurrence}
\end{center}
\end{figure}

\begin{proof}[Proof of Theorem~\ref{thm:5term}]
        We prove Theorem~\ref{thm:5term} using Pl\"ucker coordinates and relations. As a shorthand, for $W$ any set of positive integers, we write $W' = \{w + n : w \in W\}$.
        
    First, we rewrite the left side of \eqref{eq:5term} in terms of Pl\"ucker variables, using Equation \eqref{eq:PlutoMinor}.  By definition, there is only one jellyfish tableau for the ordered set partition $(A\cup B | C)$, since $|C|=r$. Call this tableau $\widehat{T}$.
    We have
    \begin{equation}\label{eq:expand_the_jellyfish_inv}
    \begin{split}
            \big[(A\cup B \mid C)\big]_r &= \sum_{T\in \rstab_r(A\cup B \mid C)}\sgn(T) \; \rspoly(T)  \\
          &=  \sgn(\widehat{T}) \;  M_{R_{A \cup B}(\widehat{T})}^{A \cup B} \cdot M_{R_C(\widehat{T})}^{C} \\
          &=  \sgn(\widehat{T}) \; (-1)^{|A \cup B|} \Delta_{([n]\setminus R_{A \cup B}(\widehat{T}))  A' B'}(\Phi(M)) \cdot (-1)^{|C|}\Delta_{([n]\setminus R_C(\widehat{T}))  C'}(\Phi(M))\\
          &= (-1)^n \sgn(\widehat{T}) \;  \Delta_{[n-r + 1, n]  A' B'}(\Phi(M)) \cdot \Delta_{[r+1, n]  C'}(\Phi(M))
    \end{split}
      \end{equation}
Here, the first and second equalities are by definition, while the third is using \eqref{eq:PlutoMinor} and the fourth line collects the sign.

We will apply the classical Pl\"ucker relation in a form that we borrow from \cite[Eqn.~9.1]{Fulton:YoungTableaux}:
\begin{equation}\label{eq:Fulton}
\Delta_{i_1,i_2,\ldots,i_d}\cdot\Delta_{j_1,j_2,\ldots,j_d}=\sum\Delta_{i_1',i_2',\ldots,i_d'}\cdot\Delta_{j_1',j_2',\ldots,j_d'},
\end{equation}
where the sum is over all pairs obtained by interchanging a fixed set of $k$ of the subscripts $j_1,j_2,\ldots,j_d$ with $k$ of the subscripts in $i_1,i_2,\ldots,i_d$, maintaining the order in each. Note that whenever the resulting subscripts are rearranged, signs are introduced. In our setting, we will have $k = |B|$. 

Before applying the Pl\"ucker relation, we rearrange the last line of \eqref{eq:expand_the_jellyfish_inv} in order to better match \eqref{eq:Fulton}, obtaining:
\begin{equation}\label{eq:expand_the_jellyfish_inv_some_more}
\begin{split}
        \big[(A\cup B \mid C)\big]_r 
        &= (-1)^n \sgn(\widehat{T}) \;   \Delta_{[r+1, n]  C'}(\Phi(M)) \cdot \Delta_{[n-r + 1, n]  A'B'}(\Phi(M)) \\
        &= (-1)^n \sgn(\widehat{T}) \;   \Delta_{[n-r+1, n]  C'  [r+1,n-r]}(\Phi(M)) \cdot \Delta_{[n-r + 1, n]  A' B'}(\Phi(M)). 
        \end{split}
        \end{equation}
In the second equality above, we also moved the interval $[r+1,n-r]$ to the end of its Pl\"ucker coordinate to simplify things later. This move does not affect the global sign because each element in $[r+1,n-r]$ moves past the $r$ elements of $[n-r+1,n]$ and the $r$ elements of $C'$, for an even total number of swaps.

The particular Pl\"ucker relation we apply is the one in which $B'$ is the fixed set of $k$ subscripts that gets interchanged. What happens is that the first $\ell$ indices in $B'$ interchange with a subset $Z$ of $[r+1,n]$ and the rest interchange with a subset $S'$ of $C'$. We can match each term from the right side of \eqref{eq:Fulton} with the column-permuted jellyfish tableau that is obtained from $\widehat{T}$ by moving the elements of $B$ from column $1$ to column $2$, moving the elements of $S$ from column $2$ to column $1$, and so that the elements of $Z$ are the row indices of the empty boxes in column $1$. By the definition of jellyfish tableaux, this also determines the row indices of the empty boxes in column $2$. See Figure~\ref{fig:jellyfish_sign_calc} for an example of this construction (the unexplained notations in that figure will be introduced later for a sign calculation). By Lemma~\ref{lem:row_col_swap}, the ordering of the labels within each column is immaterial; however, the ordering we impose here will be useful later. Thus, each choice of subsets $S$ and $Z$, corresponds to a unique jellyfish tableau contributing to the right hand side of \eqref{eq:5term} by sorting the columns in increasing order. Conversely, one can check that each of the jellyfish tableaux corresponds to a term of the Pl\"ucker relation.

\begin{figure}[htb]
    \centering
    \begin{tikzpicture}
    \node (T1) {\ytableaushort{{A_1}Y,{B_1}S,a,a,b,a,b,b,a,b}};
    \node[right= 18em of T1] (T2) {\ytableaushort{{A_1}Y,S{B_1},a,a,\none b,a,\none b, \none b,a, \none b}};
    \node[right=4.5em of T1] (fake1) {};
    \node[right=7em of fake1] (fake2) {};
    \node[below = 3em of T1] (P1) {$\Delta_{[n-r+1,n] Y'S'[r+1,n-r]} \Delta_{[n-r+1,n]A_1'B_1'(A_2' \shuffle B_2')} $};
    \node[below = 3em of T2] (P2) {$\Delta_{[n-r+1,n]Y'B'_1(Z^C \shuffle B'_2)} \Delta_{[n-r+1,n]A'_1S'(A'_2 \shuffle Z)}$};
    \draw[thick, ->] (fake1) -- (fake2);
    \draw[thick, ->] (P1) -- (P2);
    \end{tikzpicture}
    \caption{An illustration of the proof of Theorem~\ref{thm:5term}.}
    \label{fig:jellyfish_sign_calc}
\end{figure}

Theorem~\ref{thm:5term} is then proved, except for comparing the signs from the jellyfish tableaux with the signs arising from sorting indices in \eqref{eq:Fulton}. 

As a warm up, we first consider the case $|A \cup B| = r$. Note that in this case, $n=2r$, so $[r+1,n-r]=\emptyset$.  We can assume our  Pl\"ucker monomial from \eqref{eq:expand_the_jellyfish_inv_some_more} looks like $\Delta_{[n-r+1,n]  Y'S'}\Delta_{[n-r+1,n]  A'B'}$, where $Y = C \setminus S$. Since the elements of the subscripts on these Pl\"ucker variables are out of order, there is an associated sign: $-1$ to the number of inversions. We want to see how this sign changes when we rearrange indices to obtain $\Delta_{[n-r+1,n]  Y'B'} \Delta_{[n-r+1,n]  A'S'}$.  Note that for each choice of $S$, we are choosing a different ordering in the initial Pl\"ucker monomial from \eqref{eq:expand_the_jellyfish_inv_some_more}, but that this is acceptable, since we only need to understand the relation between the signs of the two monomials of interest. In comparing the signs of these monomials, it is straightforward to see that the signs differ by $-1$ to the number of pairs $(t_1, t_2) \in (B \cup S) \times (Y \cup A)$ with $t_1 < t_2$.

We now perform the corresponding sign calculation for the corresponding jellyfish tableaux.
By Lemma~\ref{lem:row_col_swap}, we can instead compare the column-permuted jellyfish tableaux 
\ytableaushort{A Y, B S} and \ytableaushort{A Y, S B} , where the elements of each set are arranged vertically in a column in some fixed order. Since inversions in a column do not contribute to the sign of a jellyfish tableau, here there is a sign change for each pair $(t_1, t_2) \in (B \cup S) \times (Y \cup A)$ with $t_1 < t_2$, matching the sign changes from shuffling indices in the Pl\"ucker monomials, as described in the previous paragraph; in addition, there is a sign change for each of the bottom $|B| = |S|$ rows, as desired to match the sign appearing on the right-hand side of \eqref{eq:5term}. This completes the proof in the special case $|A \cup B| = r$.

Now we consider the general case of this sign calculation. 
Let $B_1$ be the first $|S|$ elements of $B$ and let $B_2 = B \setminus B_1$. Similarly, let $A_1$ be the first $r - |S|$ elements of $A$ and let $A_2 = A \setminus A_1$.

We study the column-permuted version of $\widehat{T}$ for which the left column has $A_1$ (in increasing order) above $B_1$ (in increasing order) occupying the body of the jellyfish, and then $A_2$ and $B_2$ in the following rows, shuffled together so that $B_2$ occupies the rows indexed by $Z$. Write $A_2 \shuffle B_2$ for this ordering of $A_2 \cup B_2$. We also write $Z^C$ for $[r+1,n-r] \setminus Z$. The right column of our column-permuted $\widehat{T}$ has $Y = C \setminus S$ (in increasing order) above $S$ (in increasing order). See the left tableau of Figure~\ref{fig:jellyfish_sign_calc} for an illustration.

We can assume our corresponding Pl\"ucker monomial from \eqref{eq:expand_the_jellyfish_inv_some_more} is written in a similar way, as \[
\Delta_{[n-r+1,n]  Y'S' [r+1,n-r]} \Delta_{[n-r+1,n]  A_1'B_1' (A_2' \shuffle B_2')},
\]
where $Y = C \setminus S$. Then we want to compare the sign of this with the sign for
\[
\Delta_{[n-r+1,n]Y'B'_1(Z^C \shuffle B'_2))} \Delta_{[n-r+1,n]A'_1S'(A'_2 \shuffle Z))}.
\]
Again, note that for each choice of $S$, we are choosing a different ordering in the corresponding Pl\"ucker monomial. 
In comparing the signs of these monomials, it is again straightforward to see that swapping $B_1'$ with $S'$ results in a sign difference of $-1$ to the number of pairs $(t_1, t_2) \in (B_1 \cup S) \times (Y \cup A_1)$ with $t_1 < t_2$ together with $-1$ to the number of pairs $(t_1, t_2) \in (B_1 \cup S) \times A_2$ with $t_1 > t_2$. 
Furthermore, the sign difference in swapping $B_2'$ and $Z$ comes from the fact that all the elements of $B_2'$ are larger than the elements of $Z^C$ and all the elements of $Z$ are smaller than the elements of $A_2'$. So after the swap, all the elements of $B_2'$ have inversions with  the elements of $[r+1,n-r]$ to their right, while all the elements of $Z$ have inversions with the elements of $A_2'$ to their left. Since the positions occupied by $B_2'$ and $Z$ are corresponding, this results in a total of $n-2r$ inversions. Since $2r$ is even, this results in a contribution of $(-1)^n$.

We now compute the corresponding sign difference between the corresponding column-permuted jellyfish tableaux. See Figure~\ref{fig:jellyfish_sign_calc} for a visual comparison of these tableaux. It is again straightforward to see that swapping $B_1$ with $S$ results in a sign difference of $-1$ to the number of pairs $(t_1, t_2) \in (B_1 \cup S) \times (Y \cup A_1)$ with $t_1 < t_2$ together with $-1$ to the number of pairs $(t_1, t_2) \in (B_1 \cup S) \times A_2$ with $t_1 > t_2$. Furthermore, there is a sign change for each of the bottom $|B_1| = |S|$ rows.

So, continuing the calculation from \eqref{eq:expand_the_jellyfish_inv_some_more}, we have that 
\begin{align*}
   [(A\cup B \ | \ C)]_r =& (-1)^n\sgn(\widehat{T})\sum_{S\subseteq C}\sgn(\widehat{T})(-1)^n(-1)^{|S|}\big[( A\cup S \mid B\cup (C\setminus S) ) \big]_r \\
    =& \sum_{S\subseteq C}(-1)^{|S|}\big[( A\cup S \mid B\cup (C\setminus S) ) \big]_r
\end{align*}
as desired.
\end{proof}

We now note that we can use this rule in set partitions with more than two blocks by fixing all blocks except those involving $A$, $B$, and $C$.
\begin{corollary}\label{cor:twoblocks}
Partition $\{1,\ldots n\}$ into $d+1$ nonempty sets: $\pi_1,\pi_2,\ldots,\pi_{d-2}, A, B$, and $C$, where $|C|=r$. Then 
\begin{equation}
\label{eq:manyterm}
\big[(\pi_1 \mid \ldots \mid \pi_{d-2} \mid A\cup B \mid C)\big]_r=\sum_{S\subseteq C}(-1)^{|S|}\big[(\pi_1 \mid \ldots \mid \pi_{d-2} \mid A\cup S \mid B\cup (C\setminus S))\big]_r.
\end{equation}
\end{corollary}
\begin{proof}
The arguments in the proof of Theorem~\ref{thm:5term} apply, since the Pl\"ucker coordinates corresponding to $\pi_1,\pi_2,\ldots,\pi_{d-2}$ factor out.
\end{proof}

\subsection{\texorpdfstring{$r$}{r}-Jellyfish invariants for flamingo Specht modules}
\label{sec:rest}
In this subsection, we give an alternate proof of Theorem~\ref{thm:inspecht} using the recurrence of Theorem~\ref{thm:5term}. We also describe the action of the symmetric group $\mathfrak{S}_n$. Finally, we show linear independence of $r$-jellyfish invariants for $r>1$.

\begin{proof}[Alternate proof of Theorem~\ref{thm:inspecht}]
Recall from Section~\ref{sec:Spechtbackground} that $S^{(d^r,1^{n- rd})}$ is generated by fillings of partition shape $(d^r,1^{n-r d})$ where each of $\{1,\ldots,n\}$ is used exactly once, and the invariant is given by multiplying the Pl\"ucker variables corresponding to the columns of the filling. Note that if $ \pi \in \ordsetpart(n,d,r)$ has $d-1$ blocks of size $r$, then $[\pi]_r$ is a single term where the corresponding jellyfish tableau is a filling of partition shape $(d^r,1^{n-r d})$ with each of $\{1,\ldots,n\}$ used exactly once. Thus clearly $[\pi]_r\in S^{(d^r,1^{n-r d})}$ by the explicit description of Specht modules from Section~\ref{sec:Spechtbackground}. We use this fact repeatedly in our argument below.

We first consider the case where $d=2$. That is, we show that the invariant $[\pi]_r$ of any ordered set partition $\pi$ of $n$ with 2 blocks and parts of size at least $r$ is in $S^{(2^r,1^{n-2r})}$. We show this by induction on the number of elements in the smaller block.

The first non-trivial case is when the smaller block has $r+1$ elements. Suppose that $\pi=(A \mid T)$ with $T=\{t_1,t_2,\ldots,t_{r+1}\}$ and $|A|\geq r+1$. We apply Equation~\eqref{eq:5term} with $A$ as given, $B=\{t_1\}$, and $C=T\setminus \{t_1\}$. Then Equation~\eqref{eq:5term} yields
\begin{align*}
\big[(A\cup \{t_1\} \mid \{t_2,\ldots,t_{r+1}\})\big]_r&=\sum_{S\subseteq T\setminus \{t_1\}}(-1)^{|S|}\big[(A\cup S \mid \{t_1\}\cup (T\setminus \{t_1\})\setminus S)\big]_r\\
&=\big[(A \mid T)\big]_r-\sum_{i\in\{2,\ldots,r+1\}}\big[(A\cup \{t_i\} \mid T\setminus \{t_i\})\big]_r\\
&+\sum_{\substack{S\subseteq T\setminus \{t_1\} \\ |S|>1}}(-1)^{|S|}\big[(A\cup S \mid T\setminus S)\big]_r.
\end{align*}
Then the term on the left side of the equality and all terms in $\sum_{i\in\{2,\ldots,r+1\}}\big[(A\cup \{t_i\} \mid T\setminus \{t_i\})\big]_r$ are invariants for ordered set partitions comprised of one block of size $r$ and one block of size $n-r$, so these are already in $S^{(2^r, n-2r)}$. The terms of the sum 
\[\sum_{\substack{S\subseteq T\setminus \{t_1\} \\ |S|>1}}(-1)^{|S|}\big[(A\cup S \mid T\setminus S)\big]_r
\]
are all $0$ since $|T\setminus S|<r$ for $|S|>1$. Thus, we have an expression for $[\pi]_r=[(A \mid T)]_r$ in terms of elements of the Specht module. 

Fix some positive integer $m< \lfloor\frac{n}{2}\rfloor$. Suppose now that the invariant $[\pi]_r$ of any two-block ordered set partition with the smaller block of size $k$ and the larger block of size $n-k$ is in $S^{(2^r,1^{n-2r})}$ for $k\leq m$, and we will show the result holds for two-block ordered set partitions with the smaller block of size $m+1$ and the larger block of size $n-m-1$. 
Let $\pi=(A \mid T)$ be such an ordered set partition, with $T=\{t_1,t_2,\ldots,t_{m+1}\}$. 
Apply Equation~\eqref{eq:5term} with $A$ as given, $B=\{t_1\}$, and $C=T\setminus \{t_1\}$.
\begin{align*}
\big[(A\cup \{t_1\} \mid \{t_2,\ldots,t_{m+1}\})\big]_r&=\sum_{S\subseteq T\setminus \{t_1\}}(-1)^{|S|}\big[(A\cup S \mid \{t_1\}\cup (T\setminus \{t_1\})\setminus S)\big]_r\\
&=\sum_{S\subseteq T\setminus \{t_1\}}(-1)^{|S|}\big[(A\cup S \mid T\setminus S ) \big]_r
\end{align*}
Then the term on the left side of the equality has a block of size $m$, so its invariant is in $S^{(2^r,1^{n-2r})}$ by induction. On the right side of the equality, when $S=\emptyset$, $(A\cup S\mid T\setminus S)=\pi$, the ordered set partition whose invariant $[\pi]_r$ we wish to show is in $S^{(2^r,1^{n-2r})}$. When $|S|>m-r+1$, $T\setminus S$ is of cardinality less than $m+1-(m-r+1)=r$, so $\big[(A\cup S\mid T\setminus S)\big]_r=0$. When $1\leq |S|\leq m-r+1\leq m$, the cardinality of $T\setminus S$ is at most $m$, so $\big[(A\cup S,T\setminus S)\big]_r$ is in $S^{(2^r,1^{n-2r})}$ by induction. Thus, we have an expression for $[\pi]_r=[(A\mid T)]_r$ in terms of elements of the Specht module. 

This completes the proof that $[\pi]_r$ is in the Specht module $S^{(2^r,1^{n-2 r})}$ for all $\pi\in\ordsetpart(n,2,r)$. The theorem then follows from Corollary~\ref{cor:twoblocks}.
\end{proof}

For any permutation $w \in \mathfrak{S}_n$ and any $B \subseteq \{1, \dots, n\}$, define $w \cdot B = \{ w(b) : b \in B \}$.
For any ordered set partition $\pi = (\pi_1\mid  \dots\mid \pi_d ) \in \ordsetpart(n,d,r)$, let $w \cdot \pi$ be the ordered set partition with blocks $( w \cdot \pi_1\mid w \cdot \pi_2\mid \dots\mid w \cdot \pi_d )$. Note that here we are permuting the elements of the partition, while in Lemma~\ref{lem:colswap}, we were permuting the blocks of the partition.

\begin{proposition}
\label{prop:any_perm}
For any ordered set partition $\pi \in \ordsetpart(n,d,r)$ and any permutation $w \in \mathfrak{S}_n$, we have 
\[
w \cdot [\pi]_r = \sgn(w) [w \cdot \pi]_r,
\]
where $\sgn(w)$ denotes the sign of the permutation $w$.
\end{proposition}

\begin{proof}
Let $\pi=(\pi_1\mid \dots\mid \pi_d)\in\ordsetpart(n,d,r)$ and $w\in \mathfrak{S}_n$. By induction, it is enough to show the result for $w$ a simple transposition $s_i$. 

Consider $s_i$ acting on $\pi$. There are two cases: Either $i$ and $i+1$ are in the same block of $\pi$ or else they are not. 
 
First, suppose they are not in the same block. The $\pi$ and $s_i \cdot \pi$ differ by exchanging that pair of elements between blocks. For $T\in\rstab_r(\pi)$, let $\widetilde{T}$ denote the tableau obtained from $T$ by applying the permutation $s_i$, swapping the labels $i$ and $i+1$.
Observe that $\{ \widetilde{T} : T \in \rstab_r(\pi) \} = \rstab_r(s_i \cdot \pi)$.
Note further that $\sgn(T) = -\sgn(\widetilde{T})$ for all $T \in \rstab_r(\pi)$. 
Recall from Definition~\ref{def:polynomial} that \[[\pi]_r = \sum_{T\in \rstab_r(\pi)}\sgn(T) \; \rspoly(T).\]
Therefore, 
\[  s_i \cdot [\pi]_r  =   s_i \cdot \sum_{T\in \rstab_r(\pi)}\sgn(T) \; \rspoly(T) = \sum_{T\in \rstab_r(\pi)} \sgn(T) \; \rspoly(\widetilde{T}) = \sum_{\tilde{T}\in \rstab_r(s_i \cdot \pi)} -\sgn(\widetilde{T}) \; \rspoly(\widetilde{T}) = -[s_i \cdot \pi]_r,\] which yields the desired result in this case.

Now suppose $i$ and $i+1$ are in the same block of $\pi$. Then $\pi = s_i \cdot \pi$, so $[\pi]_r = [s_i \cdot \pi]_r$. On the other hand, $s_i \cdot [\pi]_r$ differs from $[\pi]_r$ by swapping two adjacent columns of exactly one determinant in each summand. Hence, $s_i \cdot [\pi]_r = - [\pi]_r = -[s_i\cdot\pi]_r$.
\end{proof}

\begin{corollary}
\label{prop:rotation}
Up to signs, the long cycle $c_n = n12\dots (n-1)$ acts by rotation and the long element $w_0$ acts by reflection.

Precisely, for any ordered set partition $\pi \in \ordsetpart(n,d,r)$, we have 
\[
c_n \cdot [\pi]_r = (-1)^{n-1} [\mathtt{rot}(\pi)]_r \quad \text{and} \quad w_0 \cdot [\pi]_r = (-1)^{\binom{n}{2}} [\mathtt{refl}(\pi)]_r,
\]
where $\mathtt{rot}$ denotes counterclockwise rotation by $(360/n)^\circ$ and $\mathtt{refl}$ denotes reflection across the diameter with endpoint halfway between vertices $n$ and $1$.
\end{corollary}
\begin{proof}
This follows from Proposition~\ref{prop:any_perm} by noting $\sgn(c_n)=(-1)^{n-1}$ and $\sgn(w_0)=(-1)^{\binom{n}{2}}$.
\end{proof}

Recall from Section~\ref{sec:set_partitions} that $\nc(n,d,r)$ denotes the set of noncrossing set partitions of $[n]$ into $d$ blocks all of size at least $r$.

\begin{theorem}\label{thm:linind}
Let $r> 1$. For each noncrossing set partition $\gamma \in \nc(n,d,r)$, order the blocks in any way to create a corresponding ordered set partition $\pi_\gamma$. Then the set $\{ [\pi_\gamma]_r : \gamma \in \nc(n,d,r) \}$ is linearly independent.
\end{theorem}
\begin{proof}
Let $k = n - r d + r$, and let $\gamma\in\nc(n,d,r)$. Order the monomials in $[\pi_\gamma]_r$ under the lexicographic order with 
\begin{align*} &x_{1,1} > x_{1,2} > \cdots > x_{1,n}  \\
> &x_{2,n} > x_{2,n-1} > \cdots >x_{2,1}\\
> &x_{3,1}>x_{3,2} >\cdots>x_{3,n}\\
&\phantom{asasdfdsf} \vdots\\
> &x_{k,1}>x_{k,2} >\cdots > x_{k,n}.
\end{align*}
(Note that the second row here is ordered differently from the others.)

Consider the leading monomial of $[\pi_\gamma]_r$ under this term order. The factors of the form $x_{1,i}$ in the leading monomial come from the smallest element of each block, and the factors of the form $x_{2,i}$ come from the largest element of each block. This information uniquely determines the noncrossing set partition $\gamma$.
Thus, each element of the set $\{ [\pi_\gamma]_r : \gamma \in \nc(n,d,r) \}$ is a polynomial where each term has a different leading monomial. 

Suppose $\{\gamma^1,\ldots,\gamma^m\}\subseteq \nc(n,d,r)$ and \[0=a_1[\pi_\gamma^1]_r+\cdots+a_m[\pi_\gamma^m]_r.\]
One of these $[\pi_\gamma^i]_r$ has the largest leading monomial under the lexicographic ordering, and none of the other invariants contain this monomial. It follows inductively that $a_1=\dots=a_m=0$.
\end{proof}

\subsection{Hook Specht modules}\label{sec:hook}

In this subsection, we give a method to construct a diagrammatic basis for the \emph{hook Specht module} $S^{(d,1^{n-d})}$. The following lemma will be useful.

\begin{lemma}\label{lem:r=1_exists}
For each $P\subseteq \{2,\ldots,n\}$ with $|P| = d-1$, there exists a noncrossing ordered set partition $\pi = (\pi_1| \dots | \pi_d) \in \ncop(n,d,1)$ such that the smallest element in each block $\pi_i$ lies in the set $\{1\}\cup P$.
\end{lemma}
\begin{proof}
We can construct such a $\pi$ as follows. Let $\{1\}\cup P = \{1 = p_1 < p_2 < \dots < p_d \}$ and let $p_{d+1} = n+1$. For $1 \leq i \leq d$, let $\pi_i = \{p_{i}, p_i + 1, \dots, p_{i+1} - 1 \}$. Thus, $\pi$ is a noncrossing ordered set partition with $d$ blocks by construction.
\end{proof}

For each subset $P \subseteq \{2,\ldots,n\}$ of size $d-1$, choose a noncrossing ordered set partition $\pi^P = (\pi_1 | \dots | \pi_d) \in \ncop(n,d,1)$ such that the smallest element in each block $\pi_i$ lies in the set $\{1\}\cup P$. Such a partition exists by Lemma~\ref{lem:r=1_exists}, although it may not be unique. For each $P$, we may make an arbitrary choice of $\pi^P$. Let $\mathcal{H}  = \{\pi^P : P \subseteq \{2,\ldots,n\} \text{ and } |P| = d-1 \}$, a set of $\binom{n-1}{d-1}$ noncrossing partitions.

\begin{theorem}\label{theorem:requals1}
Construct a subset $\mathcal{H}$ of $\ncop(n,d,1)$ as described above. Then the set $\{[\pi]_1 : \pi \in \mathcal{H}\}$ of jellyfish invariants forms a basis for the hook Specht module $S^{(d,1^{n-d})}$. 
\end{theorem}
\begin{proof}
By Theorem~\ref{thm:inspecht}, each $[\pi]_1$ is in the hook Specht module $S^{(d,1^{n-d})}$. Using the lexicographic ordering of monomials given in the proof of Theorem~\ref{thm:linind}, each jellyfish invariant $\{[\pi]_1 : \pi\in\mathcal{H}\}$ has a different leading term. It follows that this set of polynomials is linearly independent. The result then follows from the observation that the dimension of $S^{(d,1^{n-d})}$ is the number of standard Young tableaux of shape $(d,1^{n-d})$, which is easily seen to be $\binom{n-1}{d-1}=|\mathcal{H}|$.
\end{proof}

We note that Theorem~\ref{thm:5term} specializes to a three-term recurrence in the case $r=1$. 

\begin{corollary}
\label{cor:3term}
Partition $\{1,\ldots n\}$ into three nonempty sets: $A$, $B$, and $C$ where $|C|=1$. Then 
\begin{equation}
\label{eq:3term}
\big[(A\cup B \mid C)\big]_1+\big[(A\cup C \mid B)\big]_1+\big[(B\cup C \mid A)\big]_1=0.
\end{equation}
\end{corollary}
\begin{proof}
Theorem~\ref{thm:5term} with $r=1$ gives
$\big[(A\cup B \mid C)\big]_1=\big[(A\mid B\cup C)\big]_1-\big[(A\cup C\mid B)\big]_1$. Now, using Lemma~\ref{lem:row_col_swap}, switching the order of the two blocks yields a negative sign. So \[
\big[(A\mid B\cup C)\big]_1=-\big[(B\cup C \mid A)\big]_1.
\]
Bringing all terms to one side completes the proof.
\end{proof}

Using this recurrence, we obtain relations that look similar to the skein relations of the $r=2$ case. Since there are more noncrossing partitions than basis elements, one can resolve crossings in multiple ways. We can also use the recurrence to see diagrammatically (cf.\ Figure~\ref{fig:lin_dep}) the linearly dependence of noncrossing set partitions.

\begin{figure}[htbp]
    \centering
    \includegraphics[width=\textwidth]{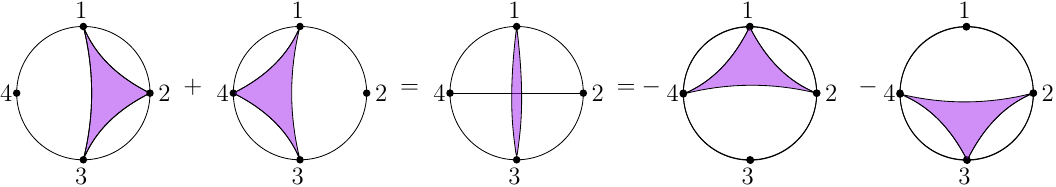}\\
    \vspace{.1in}
\includegraphics[scale=.9]{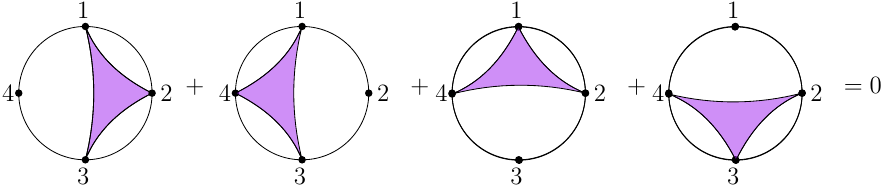}
    \caption{The top image shows how to resolve a crossing in two different ways when $r=1$. The second image shows a linear dependence among noncrossing partitions when $r=1$. In both images, the block colored purple is the first block of the ordered set partition.}
    \label{fig:lin_dep}
\end{figure}

\subsection{Final remarks on relations among jellyfish invariants}\label{sec:final}

We can use the recurrence of Theorem~\ref{thm:5term} to obtain further relations between jellyfish invariants of ordered set partitions, such as the following, where the first block is purple and the second block is white.

\begin{center}
    \includegraphics[width=\textwidth]{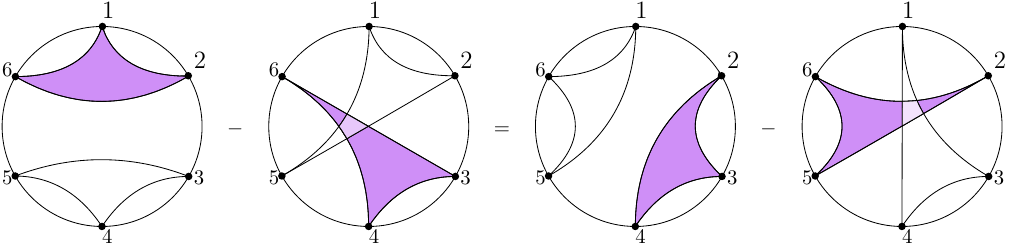}
\end{center}

Towards constructing web bases for general flamingo Specht modules, we conjecture the following extension of Theorem~\ref{thm:linind}.
For $r = 3$, the set in question is the set of noncrossing partitions, and hence follows from Theorem~\ref{thm:linind}. We have verified the conjecture for some higher $r$ by computer calculation.
\begin{conjecture}\label{conj:more_linind}
Let $r \geq 3$ and let $\mathscr{S} \subset \setpart(n,d,r)$ denote the collection of set partitions that can be made noncrossing by the application of at most $r-3$ adjacent transpositions. For each set partition $\gamma \in \mathscr{S}$, order the blocks in any way to create a corresponding ordered set partition $\pi_\gamma$. Then the set $\{ [\pi_\gamma]_r : \gamma \in \mathscr{S} \}$ is linearly independent.
\end{conjecture}

We suspect that the recurrence of Theorem~\ref{thm:5term} will be useful in approaching Conjecture~\ref{conj:more_linind}, as it allows one to rewrite the polynomial for an ordered set partition by a linear combination of polynomials for ordered set partitions that are closer to noncrossing.\footnote{After this preprint was posted, Jesse Kim built upon the ideas of this manuscript and found a basis for flamingo Specht modules in terms of jellyfish invariants of \emph{$r$-weakly noncrossing set partitions} \cite{Kim:3row}.}

It would be convenient if for $r>1$ and $\pi \in \ordsetpart(n,d,r)$, the set of $r$-jellyfish polynomials of rotations of $\pi$ were linearly independent. However, this is not true. For example, consider $\pi = ( 1 \ 2 \ 3 \ 5 \ | \ 4 \ 6 )$ and $r=2$. Then the rotation orbit of $\pi$ consists of $6$ distinct set partitions. However, one may check that the corresponding $2$-jellyfish invariants span only a $5$-dimensional space.

\section*{Acknowledgements}
We thank Ben Elias, Christian Gaetz, Stephan Pfannerer, Brendon Rhoades, Joshua P.\ Swanson, and Bruce Westbury for helpful conversations. We are grateful to Ardea Thurston-Shaine for contributing the illustrations of flamingos and jellyfish in Figures~\ref{fig:flamingo} and~\ref{fig:cute_jellyfish}. We are also very grateful to the referees, who made many  helpful suggestions including the great idea to prove Theorem~\ref{thm:5term} via Pl\"ucker relations.

OP acknowledges support from NSERC Discovery Grant RGPIN-2021-02391 and Launch Supplement DGECR-2021-00010. JS acknowledges support from Simons Foundation/SFARI grant (527204, JS) and NSF grant DMS-2247089.

\bibliographystyle{amsalphavar}
\bibliography{increasingwebs}
\end{document}